\documentclass[oneside,leqno,12pt]{amsart}
\usepackage{amssymb,amsmath,latexsym}
\usepackage{amsfonts}
\usepackage{amscd}
\usepackage{pb-diagram}
\usepackage[all]{xy}

\def\ga{\mathfrak{a}}

\def\fa{\mathfrak{a}}
\def\fk{\mathfrak{k}}

\def\fu{\mathfrak{u}}
\def\fm{\mathfrak{m}}
\def\fz{\mathfrak{z}}

\def\C{\mathbb{C}}

\def\N{\mathbb{N}}

\def\R{\mathbb{R}}

\def\cF{\mathcal{F}}

\def\cH{\mathcal{H}}

\def\cO{\mathcal{O}}

\theoremstyle{plain}
\numberwithin{equation}{section}
\newtheorem{theorem}{Theorem}[section]
\newtheorem{corollary}[theorem]{Corollary}

\newtheorem{lemma}[theorem]{Lemma}

\theoremstyle{definition}

\newcommand{\GL}{\mathrm{GL}}

\newcommand{\so}{\mathfrak{so}}

\newcommand{\IFM}{\bigoplus_{\mu \in \Lambda^+,\, d}V_\mu}
\newcommand{\LK}{\Lambda^+}
\newcommand{\ip}{\langle\, \cdot\, ,\, \cdot\, \rangle}
\def\sideremark#1{\ifvmode\leavevmode\fi\vadjust{\vbox to0pt{\vss \hbox to 0pt{\hskip\hsize\hskip1em\vbox{\hsize2cm\tiny\raggedright\pretolerance10000 \noindent #1\hfill}\hss}\vbox to8pt{\vfil}\vss}}}

\begin{document}
\title[The Segal-Bargmann Transform on Compact Symmetric
spaces]{The Segal-Bargmann Transform on Compact Symmetric Spaces and their Direct Limits}
\author{Gestur \'{O}lafsson and Keng Wiboonton}
\thanks{Both authors were supported by NSF grant DMS-0801010}
\address{Department of Mathematics, Louisiana State University, Baton Rouge,
LA 70803, USA}
\email{olafsson@math.lsu.edu}
\address{Department of Mathematics, Louisiana State University, Baton Rouge,
LA 70803, USA}
\email{kwiboo1@math.lsu.edu}

\begin{abstract} We study the Segal-Bargmann transform, or the heat transform, $H_t$ for a compact symmetric space $M=U/K$. We prove that $H_t$ is a unitary isomorphism $H_t : L^2(M) \to \cH_t (M_\C)$ using representation theory and the restriction principle. We then show that the Segal-Bargmann transform behaves nicely under propagation of symmetric spaces.  If $\{M_n=U_n/K_n,\iota_{n,m}\}_n$ is a direct family of compact symmetric spaces such that $M_m$ propagates $M_n$, $m\ge n$, then this gives rise to direct families of Hilbert spaces $\{L^2(M_n),\gamma_{n,m}\}$ and $\{\cH_t(M_{n\C}),\delta_{n,m}\}$ such that $H_{t,m}\circ \gamma_{n,m}=\delta_{n,m}\circ H_{t,n}$. We also consider similar commutative diagrams for the $K_n$-invariant case. These lead to isometric isomorphisms between the
Hilbert spaces $\varinjlim L^2(M_n)\simeq \varinjlim \mathcal{H}
(M_{n\mathbb{C}})$ as well as $\varinjlim L^2(M_n)^{K_n}\simeq \varinjlim
\mathcal{H} (M_{n\mathbb{C}})^{K_n}$.
\end{abstract}
\date{}
\maketitle

\section*{Introduction}

\noindent
Denote by $h_t(x)=(4\pi t)^{-n/2}e^{-\|x\|^2/4t}$ the heat kernel on $\R^n$ and denote by $d\mu_t (x)=h_t(x)dx$ the heat kernel measure on $\R^n$. Denote by $\Delta $ the Laplace operator on $\R^n$.
The Segal-Bargmann transform $H_t$, also called the heat kernel transform, on
$L^2(\mathbb{R}^n)$ or on $L^2(\mathbb{R}^n,\mu_t)$ is defined by mapping a function $f \in L^2(%
\mathbb{R}^n)$ to the holomorphic extension to $\mathbb{C}^n$ of $%
f*h_t=e^{t\Delta}f$. The image of $L^2(\mathbb{R}^n,\mu_t)$ under the Segal-Bargmann transform is the Fock space $\cF_t (\C^n)$ of holomorphic functions $F:\C^n\to \C$ such that $(2\pi t)^{-n}\int |F(x+iy)|^2e^{-\|x+iy\|^2/2t}\, dxdy<\infty$. Thus $\cF_t(\C^n)=L^2(\R^{2n}, d\mu_{t/2}(x)d\mu_{t/2}(y))\cap \cO (\C^n)$ whereas the image of $L^2(\mathbb{R}^n)$ is $\cH_t(\C^n)=L^2(\R^{2n}, d\mu_{t/2}(y))\cap \cO (\C^n)$, also called the Fock space. This idea, in a slightly different form, was first introduced
by V. Bargmann in \cite{Bargmann}. An infinite dimensional version
was considered by I. E. Segal in \cite{Se63}. A short history of the
Segal-Bargmann transforms for $\mathbb{R}^n$ can be found in \cite{Hall2000}
and \cite{Hall2001}.

For infinite dimensional analysis one is forced to consider the heat kernel transform defined on $L^2(\mathbb{R}^n, \mu_t)$. The reason is,
that the heat kernel measure forms a projective family of probability
measures on $\R^n$ and hence $\{L^2(\mathbb{R}^n, \mu_t)\}_{n\in \N}$ forms a direct and projective
family of Hilbert spaces. Similarly, $\{\cF_t(\C^n)\}_{n\in \N}$ forms a direct and projective family of Hilbert spaces and the $\{H_{t,n}: L^2(\R^n,\mu_t)\to \cF_t(\C^n)\}$ is direct and defines a unitary isomorphism $\varinjlim L^2(\R^n,\mu_t)\to \varinjlim \cF_t (\C^n)$.

The symmetric spaces of compact and noncompact type form a natural settings for
generalizations of the heat kernel transform and
the Segal-Bargmann transform. This was first done in \cite{Hall94} where the
Segal-Bargmann transforms were extended to the compact group case and
compact homogeneous spaces $U/K$. As in the flat case, the Segal-Bargmann
transform $H_t$ on $L^2(U/K)$ is given by the holomorphic extension of $f*h_t$ to the
complexification $U_\mathbb{C}$ of $U$. The author showed that the Segal-Bargmann
transform is an unitary isomorphism from $L^2(U)$ onto $\mathcal{O}(U_\mathbb{%
C})\cap L^2(U_\mathbb{C}, \nu_t)$, where $\mathcal{O}(U_\mathbb{C})$ denotes
the space of holomorphic functions on $U_\mathbb{C}$ and $\nu_t$ is the $U$-average heat kernel on $U_\mathbb{C}$. Analogous results for compact
symmetric spaces were given by Stenzel in \cite{Stenzel}. The image of the
Segal-Bargmann transform $H_t$ is a $L^2$-Hilbert space of holomorphic functions
on the complexification $U_\mathbb{C}/K_\mathbb{C}$ of $U/K$. In \cite{Hall94} the heat kernel measure on $U_\mathbb{C}/U$ was used to define the Fock space, whereas \cite{Stenzel} uses the heat kernel measure on the
non-compact dual $G/K$ of $U/K$. Both measures coincide as can be shown by using the Flensted-Jensen duality \cite{FJ}. In \cite{HZ09} and \cite{Wiboonton} the unitarity of the Segal-Bargmann transform was proved using the restriction principle introduced in \cite{Olafsson}.

Some work has been done on constructing a heat kernel measure on the
direct limit of some complex groups. In \cite{Gordina}, Gordina constructed
the Fock space on $\text{SO}(\infty, \mathbb{C})$, using the heat kernel
measure determined by an inner product on the Lie algebra $\so (\infty, \mathbb{C})$.
Another direction is taken in \cite{HS98} where the Segal-Bargmann
transform on path-groups is considered.

In the noncompact case new technical problems arise. In particular, in the compact case, every eigenfunction of the algebra of invariant differential operators as well as the heat kernel
itself, extends to a holomorphic function on $U_\mathbb{C}/K_\mathbb{C}$.
This follows from the fact, that each irreducible representation of $U$
extends to a holomorphic representation of $U_\mathbb{C}$ with a well
understood growth. In the noncompact case this does not hold anymore. The natural complexification in this case is the \textit{Akhiezer-Gindikin domain} $\Xi\subset G_\mathbb{C}/K_\mathbb{C}$, see \cite{AG90}. Using results from \cite{ks05} it was shown in \cite{kos05} that the image of the Segal-Bargmann transform on $G/K$ can be
identified as a Hilbert-space of holomorphic functions on $\Xi$, but in this
case the norm on the Fock space is not given by a density function as in the
flat case. Some special cases have also been considered in \cite{Ha06,HM}. A
different description that also works for arbitrary positive multiplicity functions
was given in \cite{os}.

{}From the point of view of infinite dimensional
analysis, the drawback of all of those articles is that only the invariant
measure on $G/K$ is considered, so far no description of the image of the
space $L^2(G/K,\mu_t)$ under the holomorphic extension of $f*h_t$ exists, except one can describe the space in
terms of its reproducing kernel.

The first step to consider the limit of noncompact symmetric spaces was done in \cite{Sinton}. There it was shown for a special class of symmetric spaces $G_n/K_n$ that $\{L^2(G_n/K_n,\mu_t)\}_n$ forms projective family of Hilbert spaces. But no attempt was made to consider  the Segal-Bargmann transform.

Our main goal in this article is to use some ideas from the work of J. Wolf, in particular \cite{JWCompact}, to construct the Segal-Bargmann on limits of special classes of compact symmetric spaces. In \cite{OlafssonWolf} the authors introduced the concept of \textit{propagation} of symmetric spaces. The results of \cite{JWCompact} applies to this situation resulting in an isometric embedding $\gamma_{n,m}$ of $L^2(U_n/K_n)$ into $L^2(U_m/K_m)$ for $m>n$. Let $M_n=U_n/K_n$,  and $M_{n\C} =U_{n\C}/K_{n\C}$. We show, using the ideas from \cite{JWCompact} that we have an isometric embedding $\delta_{n,m} : \cH_t(M_{\C n})\to \cH_t (M_{\C m})$ such that $H_{t,m}\circ \gamma_{n,m}=\delta_{n,m}\circ H_{t,n}$. This then results in a unitary isomorphism
\[H_{t,\infty} : \varinjlim L^2(M_n)\to \varinjlim \cH_t(M_{n\C})\, .\]

In the $K_n$-invariant case, in general $j_{n,m}(L^2(M_n)^{K_n}) \nsubseteq L^2(M_m)^{K_m}$ for $m > n$. So different maps have to be considered in this case. In this article, we define an isometric embedding $\eta_{n,m} : L^2(M_n)^{K_n}\to L^2(M_m)^{K_m}$ and similarly for $\cH_t(M_{n\C})^{K_n}$ with the embeddings $\phi_{n,m}$ such that $H_{t,m}\circ \eta_{n,m}= \phi_{n,m}\circ H_{t,n}$ resulting in a unitary isomorphism of the directed limits. This is the result from \cite{Wiboonton}.

The article is organized as follows. In Section \ref{s1} we introduce the basic notation used in this article. In Section \ref{s2} we discuss needed results from representation theory and Fourier analysis related to  symmetric spaces. The Fock space $\cH_t(M)$ is introduced in Section \ref{s3}. We show that $\cH_t(M)$ is a reproducing kernel Hilbert space and determine its reproducing kernel. We also describe $\cH_t (M)$ as a sequence space. The Segal-Bargmann transform is introduced in Section \ref{s4} and we show that it is an unitary $U$-isomorphism in Theorem 4.3. In Section \ref{s5} we recall the notion of propagation of symmetric spaces introduced in \cite{OlafssonWolf}. The infinite limit is considered in the last two sections, Section \ref{s6} and Section 7.

\section{Basic Notations}\label{s1}

\noindent
Let $M$ be a \textit{symmetric space of the compact type}. Thus,
there exists a connected compact semisimple Lie group $U$ and a nontrivial
involution $\theta : U \to U$ such that $U^{\theta}_{0} \subseteq K
\subseteq U^{\theta}$ and $M=U/K$. Here, $U^{\theta}=\{u\in U\mid \theta
(u)=u\}$ denotes the subgroup of $\theta$-fixed points, and the index ${}_0$
stands for the connected component containing the identity.

To simplify the exposition we assume that $U$ is \textit{simply connected}.
In this case $U^\theta$ is connected and hence $K=U^\theta$ is connected and
$U/K$ is simply connected. The more general case can be treated by following
the ideas in Section 2 in \cite{OlafssonSchlichtkrull}.

The base point $eK \in M$ will be denoted by $o$. Denote the Lie algebra of $%
U$ by $\mathfrak{u}$. $\theta$ defines a Lie algebra involution on $%
\mathfrak{u}$ which we will also denote by $\theta$. Then $\mathfrak{u}=%
\mathfrak{k}\oplus \mathfrak{s}$, where $\mathfrak{k} = \left\{ X \in
\mathfrak{u} \mid \theta (X) = X \right\}$ and $\mathfrak{s} =\{X\in
\mathfrak{u} \mid \theta (X)=-X\}$. Note that $\mathfrak{k}$ is the Lie
algebra of $K$, $\mathfrak{s}\simeq_K T_o(M)$ via the map $X\mapsto D_X$,
\begin{equation*}
D_Xf(o)=\left.\dfrac{d}{dt}\right|_{t=0}f(\exp (tX)\cdot o)
\end{equation*}
and $T(M)\simeq U\times_{\mathrm{Ad}|_{\mathfrak{s}}}\mathfrak{s}$.

As $U$ is compact, there is a faithful representation of $U$, so we can--and
will--assume that $U$ is linear: $U\subseteq \mathrm{U} (n)\subset \mathrm{GL%
} (n,\mathbb{C})$ for some $n\in\mathbb{N}$. Then $\mathfrak{u}\subseteq
\mathfrak{u}(n)$. Define a $U$-invariant inner
product on $\mathfrak{u}$ by
\begin{equation*}
\langle X,Y\rangle =- \mathrm{Tr\, } X Y =\mathrm{Tr\, } XY^*\, .
\end{equation*}
By restriction, this defines a $K$-invariant inner product on $\mathfrak{s}$
and hence a $U$-invariant metric on $M$. We note that $\mathfrak{k}$ and $\mathfrak{s}$ are orthogonal subspaces of $\mathfrak{u}$ with respect to $\ip$.

The inner product on $\mathfrak{u}$ determines an inner product on the dual
space $\mathfrak{u}^{*}$ in a canonical way. Furthermore, these inner
products extend to the inner products on the corresponding complexifications
$\mathfrak{u}_{\mathbb{C}}$ and $\mathfrak{u}^{*}_{\mathbb{C}}$. All these
bilinear forms are denoted by the same symbol $\ip$.

Let $\mathfrak{a} \subseteq \mathfrak{s}$ be a maximal abelian subspace of
$\mathfrak{s}$. View $\fa_\C^*$ as the space of $\mathbb{C}$-linear maps $\mathfrak{a}_\mathbb{C}%
\to \mathbb{C}$. Then $\mathfrak{a}^*=\{\lambda\in\mathfrak{a}^*_\mathbb{C}
\mid \lambda (\mathfrak{a} )\subseteq \mathbb{R}\}$ and $i\mathfrak{a}%
^*=\{\lambda\in\mathfrak{a}^*_\mathbb{C} \mid \lambda (\mathfrak{a}
)\subseteq i\mathbb{R}\}$.

For $\alpha \in \mathfrak{a}^*_\mathbb{C}$, let ${\mathfrak{u}_{\mathbb{C}}}
= \{X \in \mathfrak{u}_{\mathbb{C}} \mid (\forall H\in\mathfrak{a}_%
\mathbb{C} )\,\, [H,X] = \alpha(H)X \}$. If $\mathfrak{u}_{\mathbb{C}%
\alpha}\not=\{0\}$ then $\alpha \in i\mathfrak{a}^*$ and $\mathfrak{u}_{%
\mathbb{C}\alpha }\cap \mathfrak{u}=\{0\}$. If $\mathfrak{u}_{\mathbb{C}%
\alpha} \neq \{0\}$, then $\alpha$ is called a \textit{restricted root}.
Denote by $\Sigma = \Sigma(\mathfrak{u}_\mathbb{C}, \mathfrak{a}_\mathbb{C}%
)\subset i\mathfrak{a}^*$ the set of restricted roots. Then
\[\fu_\C=\fa_\C\oplus \fm_\C\oplus \bigoplus_{\alpha\in\Sigma}\fu_{\C\alpha}\]
where $\fm=\fz_\fk (\fa)$ is the centralizer of $\fa$ in $\fk$.

The simply connected group $U$ is contained as a maximal compact subgroup in
the simply connected complex Lie group $U_\mathbb{C}\subseteq \GL (n,\C)$ with Lie algebra
$\mathfrak{u}_\mathbb{C}=\mathfrak{u}\otimes_\mathbb{R} \mathbb{C}$. Denote
by $\theta :U_\mathbb{C}\to U_\mathbb{C}$ the holomorphic extension of $%
\theta $. Let $\sigma : U_\mathbb{C}\to U_\mathbb{C}$ be the conjugation on $%
U_\mathbb{C}$ with respect to $U$. Thus the derivative of $\sigma $ is given
by $X+iY\mapsto X-iY$, $X,Y\in \mathfrak{u}$. $\sigma $ is the Cartan
involution on $U_\mathbb{C}$ with $U=U_\mathbb{C}^\sigma$. We will also
write $\overline{g}=\sigma (g)$ for $g\in U_\mathbb{C}$.

Let $K_\mathbb{C}=U_\mathbb{C}^\theta$. Then $K_\mathbb{C}$ has Lie algebra $%
\mathfrak{k}_\mathbb{C}$ and $K$ is a maximal compact subgroup of $K_\mathbb{%
C}$. $K_\mathbb{C}$ is connected as $U_\mathbb{C}$ is simply connected and $M_\mathbb{C}=U_\mathbb{C} /K_\mathbb{C}$ is a simply connected complex symmetric space. As $\sigma (K_\mathbb{C})=K_\mathbb{C}$ it follows that $\sigma $ defines a
conjugation on $M_\mathbb{C}$ with $(M_\mathbb{C}^\sigma)_o=M$. Thus $M$ is
a totally real submanifold of $M_\mathbb{C}$. In particular,

\begin{lemma}
\label{HoloRestr} If $F\in \cO ( M_\mathbb{C})$  and $F|_M=0$, then $F=0$.
\end{lemma}

Let $\mathfrak{g}= \mathfrak{k}+i\mathfrak{s}=\mathfrak{u}_\mathbb{C}%
^{\theta\sigma}$ and let $G=U_\mathbb{C}^{\theta\sigma}$ denote the analytic
subgroup of $U_\mathbb{C}$ with Lie algebra $\mathfrak{g}$. $M^d=G/K$ is a
symmetric space of the noncompact type and $M^d=(M_\mathbb{C}^{\sigma
\theta})_o$. Hence, $M^d$ is also a totally real submanifold of $U_\mathbb{C}/ K_\mathbb{C}$.
$M^d$ is called the \textit{noncompact dual} of $M$.

The following is clear (and well known) using the Cartan decomposition of $U_%
\mathbb{C}$ and $G$:

\begin{lemma}
\label{le-CartanDecomp} Let $g\in U_{\mathbb{C}}$. Then there exists a unique $u\in U$ and a unique
$X\in i\mathfrak{u}$ such that $g=u\exp X$. We have $g\in G$ if and only if $u\in K$
and $X\in i\mathfrak{s}$.
\end{lemma}

\section{$L^2$ Fourier Analysis}\label{s2}

\noindent
In this section, we give a brief overview of the representation
theory related to harmonic analysis on $M$.

Since $U$ is assumed to be simply connected, there is a one-to-one correspondence
between $\widehat{U}$, the set of equivalence classes of irreducible unitary
representations of $U$, and the semi-lattice of \textit{dominant
algebraically integral weights} on a Cartan subalgebra containing $\mathfrak{a}$.
We denote this correspondence by $\mu \leftrightarrow (\pi_\mu, V_\mu)$.
$(\pi_\mu,V_\mu)$ is \textit{spherical} if
\begin{equation*}
V_\mu^K=\{v\in V_\mu \mid (\forall k\in K)\,\, \pi_\mu (k)v =v\}\not=
\{0\}\, .
\end{equation*}
There exists an isometric $U$-intertwining operator $V_\mu \hookrightarrow L^2(M)$ if and only if $V_\mu^K \neq \{0\}$. In that case $\dim V_\mu^K=1$. The description of the highest weights of the spherical representations is given by
the Cartan-Helgason theorem, see Theorem 4.1, p. 535 in \cite{He84}. Fix a positive system $\Sigma^+\subset \Sigma$. Let
\begin{equation}
\Lambda^+_K (U)=\left\{ \mu \in i\mathfrak{a}^*\,\left|\, (\forall \alpha
\in \Sigma^+)\,\, \frac{\langle \alpha ,\mu \rangle }{\langle \alpha,\alpha
\rangle} \in \mathbb{Z}^+=\{0,1,\ldots \}\right.\right\}\, .
\end{equation}
As both $U$ and $K$ will be fixed for the moment we simply write $\Lambda^+$
for $\Lambda^+_K(U)$. $\Lambda^+$ is contained in the semi-lattice of
dominant algebraically integral weights.

\begin{theorem}
Let $(\pi_\mu,V_\mu)$ be an irreducible representation of $U$ with highest
weight $\mu$. Then $\pi_\mu$ is spherical if and only if $\mu \in \Lambda^+$.
\end{theorem}

If nothing else is said, then we will from now on assume that $(\pi_\mu,V_\mu)$ is spherical.
$\ip_\mu$ will denote a $U$-invariant inner product on $V_\mu$. The corresponding
norm is denoted by $\|\, \cdot\, \|_\mu$. Let $d (\mu) =\dim V_\mu$. Then $\mu \mapsto d (\mu )$ extends to a polynomial function on $\fa_\C^*$ of degree $\sum_{\alpha \in\Sigma^+} \dim_\C \fu_{\C\alpha}$.
We fix $e_\mu \in V_\mu^K$ with $\|e_\mu \|=1$. The function $g\mapsto \pi_\mu
(g)e_\mu$ is right $K$-invariant and defines a $V_\mu$ valued function on $M$. We write $\pi_\mu (x)e_\mu = \pi_\mu (g)e_\mu$ if $x =g\cdot o$, $g \in U$.

For $u\in V_\mu$ let
\begin{equation}  \label{eq-piu}
\pi_\mu^u(x)=\langle u,\pi_\mu (x) e_\mu\rangle_\mu\, .
\end{equation}

The representation $\pi_\mu$ extends to a holomorphic representation of $U_%
\mathbb{C}$ which we will also denote by $\pi_\mu$. As
\begin{equation}  \label{eq-pig*}
\pi_\mu (g)^*=\pi_\mu (\sigma (g)^{-1})
\end{equation}
on $U$ and both sides are anti-holomorphic on $U_\mathbb{C}$, it follows
that (\ref{eq-pig*}) holds for all $g\in U_\mathbb{C}$. We  extend $\pi^u_\mu$ to a holomorphic function on $M_\mathbb{C}$ by
\begin{equation} \label{eq-holoext-1}
\widetilde \pi_\mu^u(z) = \langle u,\pi_\mu (\sigma (z)) e_\mu\rangle_\mu
:= \langle u,\pi_\mu (\sigma (g)) e_\mu\rangle_\mu = \langle \pi_\mu (g^{-1})u, e_\mu \rangle_\mu\, ,\quad z=g\cdot o\, .
\end{equation}

We normalize the invariant measure on compact groups so that the total
measure of the group is one. Then $\int_M f(m)\, dm=\int_U f(a \cdot o)\, da$
defines a normalized $U$-invariant measure on $M$. The corresponding $L^2$%
-inner product, respectively norm, is denoted by $\langle \, \cdot\, ,\,
\cdot\, \rangle_2$, respectively $\|\, \cdot \, \|_2$.

Recall that by Schur's orthogonality relations we have
\begin{equation}\label{OrthRel}
\int_U \langle u,\pi_\mu (g)v\rangle_\mu
\langle \pi_\delta (g) x,y \rangle_\delta \, du =\delta_{\mu,\nu}\frac{1}{d(\mu )} \langle u,x\rangle_\mu\langle y,v\rangle_\mu \, .
\end{equation}
In particular, $V_\mu \to L^2(M)$, $u\mapsto d(\mu )^{1/2}\pi^u_\mu$ is a unitary $U$-isomorphism onto its image
$L^2(M)_\mu\subset L^2(M)$. Furthermore
\[L^2(M)=\bigoplus_{\mu\in\Lambda^+} L^2(M)_\mu \, .\]
Furthermore, each function $f\in L^2(M)_\mu$ has a holomorphic extension $\widetilde{f}$ to $M_\C$.

\begin{lemma}
\label{le-PropWidetildePi} Let the notations be as above.

\begin{enumerate}
\item Let $\mu,\delta\in\Lambda^+$, $u\in V_\mu$, $v\in V_\delta$, and $%
H_1,H_2\in\mathfrak{a}_\mathbb{C}$. Then
\begin{eqnarray*}
\int_U \widetilde{\pi}^u_\mu (g\exp H_1)\overline{\widetilde{\pi}^v_\delta
(g\exp H_2)}\, dg &=&\frac{\delta_{\mu , \delta}}{d(\mu )}\langle
u,v\rangle_\mu \langle \pi_\mu(\exp H_2)e_\mu,\pi_\mu (\exp
H_1)e_\mu\rangle_\mu \\
&=& \frac{\delta_{\mu , \delta}}{d(\mu )} \langle u,v\rangle_\mu \langle
e_\mu, \pi_\mu (\exp (H_1-\sigma (H_2))e_\mu\rangle_\mu.
\end{eqnarray*}
\item Let $L\subset M_\mathbb{C}$ be compact. Then there exists a constant $%
C_L>0$ such that
\begin{equation*}
|\pi^u_\mu (z)|\le e^{C_L \|\mu \|}\|u\|_\mu\,
\end{equation*}
for all $z\in L$.
\end{enumerate}
\end{lemma}

\begin{proof}
(1) This follows from Schur's orthogonality relations (\ref{OrthRel}):
\begin{eqnarray*}
\int_U \widetilde{\pi}^u_\mu (g\exp H_1)\overline{\widetilde{\pi}^v_\delta
(g\exp H_2)}\, dg&=& \int_U \langle u, \pi_\mu (g) \pi_\mu (\exp
H_1)e_\mu\rangle_\mu \langle \pi_\delta (g)\pi_\delta (\exp H_2)e_\delta
,v\rangle_\delta \, dg \\
&=& \frac{\delta_{\mu,\delta}}{d(\mu )}\langle u,v\rangle_\mu \langle
\pi_\mu (\exp H_2)e_\mu ,\pi_\mu (\exp H_1)e_\mu\rangle_\mu \\
&=& \frac{\delta_{\mu,\delta}}{d(\mu )}\langle u,v\rangle_\mu \langle e_\mu
, \pi_\mu (\exp (H_1-\sigma (H_2)))e_\mu \rangle_\mu
\end{eqnarray*}
where we used that  $\pi_\mu (\exp H_2 )^*=\pi_\mu (\sigma (\exp H_2)^{-1})$. (2) follows
from Lemma 3.9 in \cite{Branson}.
\end{proof}

For $f\in L^2(U)\subset L^1(U)$ let
\[\pi_\mu (f)=\int_U f(g)\pi_\mu(g)\, dg\]
be the integrated representation. Denote by $P_\mu =\int_K\pi_\mu
(k)\, dk : V_\mu \to V_\mu^K$ the orthogonal projection $v\mapsto \langle
v,e_\mu\rangle e_\mu$. If $f\in L^2(M)=L^2(U)^K$, then
$\pi_\mu (f)=\pi_\mu (f)P_\mu$.
Define the (vector valued) \textit{Fourier transform} of $f\in L^2(M)$ by
\begin{equation}  \label{def-FT}
\widehat{f}_\mu:= \pi_\mu (f)e_\mu\, .
\end{equation}
Denote the left-regular representation of $U$ on $L^2(M)$ by $L$. Thus $%
(L(a)f)(x)=f(a^{-1}\cdot x)$. Then
\begin{equation}  \label{eq-intertwining}
\widehat{L(a)f}_\mu  = \pi_\mu (a )\widehat{f}_\mu \, .
\end{equation}

To describe the image of the Fourier transform let
\begin{eqnarray*}
\bigoplus_{\mu \in \Lambda^+,d}V_\mu &:=& \left\{(v_\mu )_{\mu\in\Lambda^+}\, \left|\,
(\forall \mu\in\Lambda^+)\,\, v_\mu \in V_\mu\text{ and }
\sum_{\mu\in\Lambda^+} d (\mu )\|v_\mu \|^2_\mu<\infty\right.\right\} \\
&=&\left\{v:\Lambda^+\to\Pi_{\mu\in\Lambda^+} V_\mu\,\left|\, (\forall \mu\in\Lambda^+)\,\, v(\mu )\in
V_\mu\text{ and } \sum_{\mu\in\Lambda^+}d(\mu )\|v(\mu )\|^2<\infty\right.\right\}\, .
\end{eqnarray*}
Then $\bigoplus_{\mu \in \Lambda^+,d}V_\mu$ is a Hilbert space with the inner
product
\begin{equation*}
\langle (v_\mu ), (w_\mu )\rangle =\sum_{\mu \in\Lambda^+} d (\mu )\langle
v_\mu, w_\mu\rangle_{\mu}\, .
\end{equation*}
The group $U$ acts unitarily on $\bigoplus_{\mu \in \Lambda^+,d}V_\mu$ by
$(\widehat{L}(a)(u_\mu ))_\mu = (\pi_\mu (a)u_\mu )_\mu$.

\begin{theorem}
\label{th-Plancherel1} If $f\in L^2(M)$ then $(\widehat{f}_\mu)_\mu \in
\bigoplus_{\mu \in \Lambda^+,d}V_\mu$ and $\widehat{\hbox to1em{}} :
L^2(M)\to\bigoplus_{\mu \in \Lambda^+,\, d}V_\mu$ is a unitary $U$%
-isomorphism with inverse
\begin{equation*}
v\mapsto \sum_{\mu\in\Lambda^+} d(\mu )\pi_\mu^{v (\mu)}\, .
\end{equation*}
In particular if $f\in L^2(M)$ then
\begin{equation}  \label{eq-FourierSeries}
f=\sum_{\mu\in\Lambda^+}d (\mu )\langle \widehat{f}_\mu,\pi_\mu (\,\cdot
\, )e_\mu \rangle_\mu = \sum_{\mu\in\Lambda^+}d (\mu )\pi_{\mu}^{\widehat{f}_\mu} \quad\text{and}\quad \|f\|_2^2=\sum_{\mu\in\Lambda^+}
d (\mu )\|\widehat{f}_\mu\|_\mu^2
\end{equation}
where the first sum is taking in $L^2(M)$. If $f$ is smooth, then the sum
converges in the $C^\infty$-topology. Furthermore
\begin{enumerate}
\item If $f\in L^2(M)$ then the orthogonal projection of $f$ into $L^2 (M)_\mu$ is given by
$f_\mu = d(\mu )\langle \widehat{f}_\mu ,\pi_\mu
(\, \cdot \, )e_\mu \rangle$.
\item $f_\mu$ has a holomorphic continuation $\widetilde{f}_\mu$ to $M_\mathbb{C}$ which is given by
\begin{equation}  \label{eq-holExtfmu}
\widetilde{f}_\mu (z)=d(\mu ) \langle \widehat{f}_\mu,\pi_\mu (\bar{z}%
)e_\mu \rangle\, .
\end{equation}
\item
If $L\subset M_\mathbb{C}$ is compact, then there exists a constant $C_L>0$
such that
\begin{equation*}
\sup_{x\in L} |\widetilde{f}_\mu (x)|\le d(\mu ) e^{C_L \|\mu \|}\|f\|_2\, .
\end{equation*}
\end{enumerate}
\end{theorem}

\begin{proof}
This is well known but we indicate how the statements follows from the
general Plancherel formula for compact groups, see \cite{FollandHA}, p. 134.
We have
\begin{equation*}
f(x)=\sum_{\mu\in \Lambda^+} d (\mu )\mathrm{Tr\, } (\pi_\mu (x^{-1})\pi_\mu
(f)) \quad \text{and}\quad \|f\|^2_2=\sum_{\mu\in \Lambda^+} d (\mu )\mathrm{%
Tr\, } (\pi_\mu (f)^*\pi_\mu (f))\, .
\end{equation*}
Extending $e_\mu$ to an orthonormal basis for $V_\mu$, it follows from $\pi_\mu(f) =\pi_\mu (f)P_K$ that
\begin{equation*}
\mathrm{Tr\, } (\pi_\mu (x^{-1})\pi_\mu (f))=\langle \pi_\mu (x^{-1})\pi_\mu
(f)e_\mu ,e_\mu\rangle_\mu = \langle \widetilde{f}(\mu ),\pi_\mu
(x)e_\mu\rangle_\mu
\end{equation*}
and
\begin{equation*}
\mathrm{Tr\, } (\pi_\mu (f)^*\pi_\mu (f))=\langle \pi_\mu (f)^*\pi_\mu
(f)e_\mu ,e_\mu\rangle_\mu =\langle \pi_\mu (f)e_\mu , \pi_\mu
(f)e_\mu\rangle_\mu=\|\widehat{f}(\mu )\|_\mu^2\, .
\end{equation*}
The $L^2$-part of the theorem follows now from the Plancherel formula for $U$.
The intertwining property is a consequence of (\ref{eq-intertwining}). For
the last statement see \cite{Sugiura}.

That the orthogonal projection $f\mapsto f_\mu$ is given by
$f_\mu (x)=d(\mu )\langle \widehat{f}_\mu,\pi_\mu (x)e_\mu\rangle $ follows
from (\ref{eq-FourierSeries}). The last part follows from (\ref{eq-holoext-1})
and $\|\widehat{f}_\mu \|_\mu \le \|f\|_2$.
\end{proof}

The \textit{spherical function} on $M$ associated with $\mu$ is the matrix
coefficient
\begin{equation}  \label{eqSphfct} \psi_{\mu}(g) =
\pi_\mu^{e_\mu}(g)= \left\langle e_\mu, \pi_\mu(g)e_\mu \right\rangle_\mu\, ,
\ g \in U\, .
\end{equation}
It is independent of the choice of $e_\mu$ as long as $\|e_\mu\|_\mu =1$. We will
view $\psi_\mu$ as a $K$-biinvariant function on $U$ or as a $K$-invariant
function on $M$. $\psi_\mu$ is the unique element in $L^2(M)^K_\mu$ which
takes the value one at the base point $o$.

If $f\in L^2(M)^K$ then $\widehat{f}_\mu = \langle \widehat{f}_\mu ,e_\mu \rangle_\mu e_\mu\in V_\mu^K $. Furthermore,
\begin{equation*}
\langle \widehat{f}_\mu ,e_\mu \rangle_\mu = \langle \pi_{\mu}(f )e_\mu ,e_\mu \rangle_\mu
= \int_M f(m)\overline{\psi_\mu(m)}\, dm =\int_U f(a\cdot o ) \psi_\mu (a^{-1})\, da\, .
\end{equation*}
This motivates the definition of the \textit{spherical Fourier transform} on
$L^2(M)^K$ by
\begin{equation}  \label{de-FTKinv}
\widehat{f}(\mu) =\langle f,\psi_\mu \rangle_2 \, .
\end{equation}

Define the weighted $\ell^2$-space $\ell_d^2(\Lambda^+ )$ by
\begin{equation*}
\ell_d^2(\Lambda^+) := \left\{(a_\mu)_{\mu \in \Lambda^+} \ \left| \ a_\mu
\in \mathbb{C} \text{ and } \sum_{\mu \in \Lambda^+} d(\mu)|a_\mu|^2 <
\infty \right. \right\}.
\end{equation*}
Then $\ell_d^2(\Lambda^+ )$ is a Hilbert space.

\begin{theorem}
\label{SphericalPlancherel} The
spherical Fourier transform is a unitary isomorphism of $L^2(M)^K$ onto $%
\ell^2_d(\Lambda^+)$ with inverse
\begin{equation*}
(a_\mu)_\mu \mapsto \sum_{\mu \in \Lambda^+} d (\mu )a_\mu \psi_\mu
\end{equation*}
where the sum is taken in $L^2(M)^K$. It converges in the $C^\infty$%
-topology if $f$ is smooth. Furthermore,
\begin{equation*}
\|f\|_2^2=\sum_{\mu\in\Lambda^+} d(\mu ) |\widehat{f}(\mu )|^2\, .
\end{equation*}
\end{theorem}

\begin{proof}
This follows directly from Theorem \ref{th-Plancherel1}.
\end{proof}

Note that $\psi_\mu$ has a holomorphic extension $\widetilde{\psi}_\mu$ to $%
M_\mathbb{C}$ given by $\widetilde{\psi}_\mu (z)=\langle e_\mu ,\pi_\mu
(\sigma (z))e_\mu\rangle_\mu$.

\begin{lemma}
Let $f\in L^2(M)$. Then $f_\mu (z)=d(\mu )f*\widetilde \psi_\mu (z)$.
\end{lemma}

\begin{proof}
We have
\begin{eqnarray*}
f_\mu (z) &=& d(\mu )\langle \widehat{f}_{\mu},\pi_\mu (\sigma
(z))e_\mu\rangle_\mu \\
&=&\int_U f(g \cdot o)\langle \pi_\mu (g )e_\mu ,\pi_{\mu} (\sigma
(z))e_\mu\rangle_\mu \, dg \\
&=&\int_U f(g \cdot o)\langle e_\mu ,\pi_{\mu} (\sigma
(g^{-1}z))e_\mu\rangle_\mu \, dg \\
&=&f*\widetilde \psi_\mu (z)\, .
\end{eqnarray*}
\end{proof}

\section{The Fock Space $\mathcal{H}_t(M_\mathbb{C})$}\label{s3}
\noindent
In this section, we start by recalling some needed and well-known facts on integration on $M^d=G/K$,
the noncompact dual of $M$. We then introduce the heat kernel $h_t^d$ on $M^d$. For more details and proofs for the statements involving $h_t^d$ we refer to  \cite{ks05,kos05,os,os2} and the references therein.
We introduce the Fock space $\mathcal{H}_t(M_\mathbb{C})$. Using
the restriction principle introduced in \cite{Olafsson} we show that $\mathcal{H}_t (M_\mathbb{C})$ is
isomorphic to $L^2(M)$ as a $U$-representation. In the next section
we will show that the Segal-Bargmann transform $H_t: L^2(M)\to \mathcal{H}_t
(M_\mathbb{C})$ is a unitary isomorphism.

Let
\begin{equation*}
(i\mathfrak{a} )_+=\{H\in i\mathfrak{a}\mid (\forall \alpha \in \Sigma^+
)\,\, \alpha (H)>0\}\, .
\end{equation*}
The following is a well known decomposition theorem for an involution
commuting with a given Cartan involution, see \cite{FJ1} or Proposition
7.1.3. in \cite{Sch}.

\begin{lemma}
\label{decomposeinX_C} Let $z\in M_\mathbb{C}$. Then there exist $u \in U$
and $H \in i\mathfrak{a}$ such that $z = u\exp(H)\cdot o$. If $%
u_1\exp(H_1)\cdot o = u_2\exp(H_2)\cdot o$ then there exists $w\in W$ such
that $H_2 = w\cdot H_1$. If we choose $H \in \overline{(i\mathfrak{a})_+}$,
then $H$ is unique.
\end{lemma}

Let $m$ be a $U_\mathbb{C}$-invariant measure on $M_\mathbb{C}$.

\begin{theorem}
\label{integrationX(C)} We can normalize $m$ such that for $f\in L^1(M_{%
\mathbb{C}})$
\begin{equation*}
\int_{M_{\mathbb{C}}} f(z) dm(z) = \int_{U}\int_{(i\mathfrak{a})_{+}} f(
u\exp H \cdot o )J(H) dH du ,
\end{equation*}
where
\begin{equation*}
J(H) = \prod_{\alpha \in \Sigma^{+}} \sinh \left( 2\left\langle \alpha , H
\right\rangle\right)\, .
\end{equation*}
\end{theorem}

\begin{proof}
This follows from the general integration theorem for symmetric space
applied to $M_\mathbb{C}$, see \cite{FJ1} or Proposition 8.1.1 in
\cite{Sch}, using that $\sinh(2x)=2\sinh (x)\cosh(x)$.
\end{proof}

Let $m_1$ be a $G$-invariant measure on $M^d$.

\begin{theorem}
\label{integrationY} We can normalize $m_1$ such that for $f\in L^1(M)$
\begin{equation*}
\int_{M^d} f(x) dm_{1}(x) = \int_{K}\int_{(i\mathfrak{a})_+}f(k\exp H \cdot
o)J_1(H) dH dk\, .
\end{equation*}
where $J_1(H) = J(2^{-1} H)$.
\end{theorem}

Corresponding to the positive system $\Sigma^+$ there is an Iwazawa
decomposition $G=KA^dN$ of $G$, where $A^d=\exp (i\mathfrak{a})$. Write $%
x\in G$ as $x=k(x)a(x)n(x)$. For $\alpha\in\Sigma $ let $m_\alpha =\dim_{%
\mathbb{C}} \mathfrak{u}_{\mathbb{C},\alpha}$ and let $\rho =2^{-1}
\sum_{\alpha \in \Sigma^+} m_\alpha \alpha \in i\mathfrak{a}^*$. Let
\begin{equation*}
\varphi_\lambda (x)=\int_K a(gk)^{i\lambda -\rho}\, dk
\end{equation*}
denote the spherical functions on $M^d$ with spectral parameter $\lambda$,
see \cite{He84}, Theorem 4.3, p. 418, and p. 435.

\begin{lemma}
\label{le-phipsi} \label{identity} Let $\mu \in \Lambda^+$. Then $\varphi_{i(\mu+\rho)}$
extends to a holomorphic function $\widetilde{\varphi}_{i(\mu +\rho)}$ on
$M_\mathbb{C}$ and $\widetilde{\psi_\mu} = \widetilde{\varphi}_{i(\mu+\rho)}$.
\end{lemma}

\begin{proof}
See the proof of Lemma 2.5 in \cite{Branson} and the fact that
$\varphi_{\lambda}(g^{-1})=\varphi_{-\lambda }(g)$, see \cite{He84}, p. 419.
\end{proof}

Consider the complex-bilinear extension $(\, \cdot \, ,\, \cdot \, )$ of
$\langle \, \cdot \, ,\, \cdot \, \rangle|_{(i\mathfrak{a})^*\times
(i\mathfrak{a})^*}$ to $\mathfrak{a}_\mathbb{C}^*$. We write $\lambda \cdot
\mu =(\lambda ,\mu)$ and $\lambda^2=(\lambda ,\lambda)$.

The trace form $(X,Y)=-\mathrm{Tr\, } (XY^*)$ defines a $K$-invariant metric
on $i\mathfrak{s}$ and hence a $G^d$-invariant metric on $M^d$. We consider the
Laplace operator $\Delta^d$ associated with this metric. Let
$h_{t}^{d}$ be the heat kernel associated to the Laplace-Beltrami operator on
the noncompact symmetric space $M^d$. Then $h_t^d$ is $K$-invariant.
Thus $h_t\ge 0$, $\{h_t^d\}_{t>0}$ is an approximate unity and
$e^{t\Delta_d}f=h_t*f$. In particular $h_t*\varphi_\lambda = e^{-t(\lambda^2+\rho^2)}\varphi$ for every
bounded spherical function.

\begin{theorem}
\label{int-form} Let $\lambda \in \mathfrak{a}_\mathbb{C}^*$. Then $
h_{2t}\varphi_{-\lambda}\in L^1(M^d)$,  and
\begin{equation*}
\int_{M^d} h_{2t}^d (x)\varphi_{-\lambda}(x)\, dm (x) =\int_{(i\mathfrak{a})_+
}h_{2t}^d(\exp H) \varphi_{-\lambda }(\exp H)J_1(H)\, dH=
e^{-2t(\lambda^2+\rho^2)}\, .
\end{equation*}
In particular
\begin{eqnarray}
\int_{(i\mathfrak{a})_+ }h_{2t}^d(\exp H) \widetilde{\psi}_{\mu} (\exp
H)J_1(H)\, dH&=& \frac{1}{|W|}\int_{i\mathfrak{a}}h_{2t}^d(\exp H) \widetilde{\psi}%
_{\mu} (\exp H)|J_1(H)|\, dH  \notag \\
&=&e^{2t(\mu^2+2\mu\cdot \rho)} \, .  \label{eq-heatInnPpsiMu}
\end{eqnarray}
\end{theorem}

\begin{proof}
First note that for $H\in (i\mathfrak{a})_+$, we have
\begin{equation*}
J_1(H)\le C_1 e^{2\rho (H)}
\end{equation*}
and by simple use of the estimates for $\varphi_{-\lambda}(\exp H)$ from
Proposition 6.1 in \cite{Opdam}, we have
\begin{equation*}
|\varphi_{-\lambda} (\exp H)|\le C e^{\|\lambda\|\|H\|}\, .
\end{equation*}
Finally, according to the Main Theorem in \cite{AO03}, p. 33, there exists a positive
polynomial $p(H)$ on $i\mathfrak{a}$ such that
\begin{equation*}
h_{2t}^d(\exp H) \le p(H)e^{-\rho (H)}e^{-\|H\|^2/8t}\,\quad \text{ for all } H\in (i\mathfrak{a})_+\,  .
\end{equation*}
Let $L\subset \ga_\C^*$. Putting those three estimates together we get
\[|h_{2t}^d(\exp H) \varphi_{-\lambda} (\exp H)J_1(H)|\le C_1 |p(H)|e^{C_2\|H\|-\|H\|^2/8t}\]
for some constants $C_1,C_2>0$. As the right hand side is integrable it follows that
$H\mapsto h_{2t}^d(\exp H) \varphi_{-\lambda }(\exp H)J_1(H)$ is
integrable on $(i\mathfrak{a})_+$ and
\begin{equation*}
\lambda\mapsto\int_{(i\mathfrak{a})_+ }h_{2t}^d(\exp H) \varphi_{-\lambda }(\exp H)J_1(H)\, dH = \int_{M^d} h_{2t}^d (m)\varphi_{-\lambda}(m)\, dm
\end{equation*}
is holomorphic.

It is well known, see \cite{AO03}, that for $\lambda\in (i\mathfrak{a})^*$,
\begin{equation}  \label{eq-hd}
\int_{M^d} h_{2t}^{d}(x)\varphi_{-\lambda}(x) dm_{1}(x) = e^{-2t((\lambda
,\lambda)+ ( \rho ,\rho))}\, .
\end{equation}
As both sides are holomorphic it follows that (\ref{eq-hd}) holds for all $%
\lambda\in \mathfrak{a}_\mathbb{C}^*$. As the holomorphic extension is an
even function in $\lambda$, (\ref{eq-heatInnPpsiMu}) follows now from Lemma %
\ref{le-phipsi} by taking $\lambda = i(\mu +\rho)$.
\end{proof}

We define
\begin{equation}  \label{def-pt}
p_t(z) := 2^rh_{2t}^d(\exp (2H)\cdot o)\quad \text{for}\quad z= (u\exp
H)\cdot o \in M_\mathbb{C} \, ,\,\, u \in U\, ,\,\, H \in i\mathfrak{a}\, .
\end{equation}
Here $r = \dim{\mathfrak{a}}$. Define a measure $\mu_t$ on $M_\mathbb{C}$ by
$d\mu_t(z):=p_t(z)dm(z)$. We note that $p_t$ is $U$-invariant by definition.
As $m_1$ is $M_\mathbb{C}$ invariant, it follows that the measure $\mu_t$ is
$U$-invariant. Define the \textit{Fock space} $\mathcal{H}_t(M_{\mathbb{C}})$
by
\begin{equation}  \label{de-FockSp}
\mathcal{H}_t(M_{\mathbb{C}}) := \left\{F \in \mathcal{O}(M_\mathbb{C})\,
\left| \, \|F\|^2_{t} = \int_{M_{\mathbb{C}}}|F(z)|^2\, d\mu_t(z) < \infty
\right. \right\} =L^2(M_\mathbb{C},\mu_t)\cap \mathcal{O} (M_\mathbb{C})\, .
\end{equation}
Using that $\mu_t$ is $U$-invariant we get the following standard results (c.f. \cite%
{Neeb}, \cite{Faraut1}):

\begin{theorem}
\label{th-propHt} Let $t>0$. Then $\mathcal{H}_t (M_\mathbb{C})$ is an $U$%
-invariant Hilbert space of holomorphic functions. In particular, if $%
L\subset M_\mathbb{C}$ is compact, then there exists a constant $C_L>0$ such
that
\begin{equation*}
\sup_{z\in L}|F(z)|\le C_L \|F\|_t\quad \text{for all}\quad F\in\mathcal{H}%
_t(M_\mathbb{C})\, .
\end{equation*}
Furthermore, there exists a function $K_t:M_\mathbb{C}\times M_\mathbb{C}\to
\mathbb{C}$ such that

\begin{enumerate}
\item $K_t(\, \cdot \, ,w)\in \cH_t(M_\C)$ for all $w\in M_\C$.
\item $F(w)=\langle F,K_t(\cdot ,w)\rangle $ for all $F\in \mathcal{H}_t(M_\mathbb{C})$.
\item $K_t(z,w)=\overline{K_t(w,z)}$.
\item $(z,w)\mapsto K_t(z,w)$ is holomorphic in $z$ and anti-holomorphic in $w$.
\end{enumerate}
\end{theorem}

$K_t(z,w)$ is \textit{the
reproducing kernel} of $\mathcal{H}_t(M_\mathbb{C})$.
Recall that the existence of $K_{t,w}(z):= K_t(z,w)$ follows from the fact that the
point evaluation $F\mapsto F(w)$ is continuous on $\mathcal{H}_t(M_\mathbb{C})$
and hence this evaluation map is given by the inner product with a function $K_{t,w}\in\mathcal{H}_t(M_\mathbb{C})$. Then
\begin{equation*}
K_t(z,w)=\langle K_{t,w},K_{t,z}\rangle
\end{equation*}
which clearly implies (2) and (3).

\begin{lemma}
\label{focknorm} If $\mu\in\Lambda^+$ and $v\in V_\mu$ then $\widetilde{\pi}%
_\mu^v\in \mathcal{H}_t(M_\mathbb{C})$. Furthermore, if $\delta \in \Lambda^+
$, and $w \in V_{\delta}$, then
\begin{equation*}
\langle \widetilde{\pi}_\mu^v , \widetilde{\pi}_\delta^w\rangle_{t} =
\delta_{\mu,\delta} \frac{e^{2t (\mu^2 + 2\rho\cdot \rangle)}}{d(\mu)}%
\langle v, w \rangle_\mu.
\end{equation*}
\end{lemma}

\begin{proof}
We show first that $\widetilde{\pi}_\mu^v\in\mathcal{H}_t$. Clearly, $%
\widetilde{\pi}_\mu^v$ is holomorphic, so we only have to show that it is
square integrable with respect to $\mu_t$. We have by Theorem \ref%
{integrationX(C)} and (1) of Lemma \ref{le-PropWidetildePi} that
\begin{eqnarray*}
\langle \widetilde{\pi}_\mu^v , \widetilde{\pi}_\mu^v\rangle_{t} &=&
\int_{M_{\mathbb{C}}}| \widetilde{\pi}_\mu^v(z)|^2\, d\mu_t(z) \\
&=& 2^r\int_U \int_{(i\mathfrak{a})_+} |\widetilde{\pi}_\mu^v(g\exp H)|^2
h^d_{2t}(\exp(2H)) J(H)dH\ dg \\
&=& 2^r \int_{(i\mathfrak{a})_+} h^d_{2t}(\exp(2H)\cdot o) \left(\int_U
\widetilde{\pi}_\mu^v(g\exp H)\overline{\widetilde{\pi}_\mu^v(g\exp H)}\,
dg \right) J(H)dH \\
&=&\frac{2^r\|v\|^2}{d(\mu )}\int_{(i\mathfrak{a})_+} \widetilde{\psi}%
_\mu(\exp (2H)) J_1(2H)h_{2t}^d (\exp ( 2H)\cdot o)\, dH \\
&=&\frac{\|v\|^2}{d(\mu )}\int_{(i\mathfrak{a})_+} \widetilde{\psi}_\mu(\exp
(H)) J_1(H)h_{2t}^d (\exp ( H)\cdot o)\, dH \\
&=&\frac{\|v\|^2}{d(\mu )}e^{2t\langle \mu +2\rho ,\mu\rangle} \\
&<&\infty
\end{eqnarray*}
where the last equality follows from Theorem \ref{int-form}.

Now using again Theorem \ref{integrationX(C)}, Lemma \ref{le-PropWidetildePi}, Theorem \ref{int-form}, and Fubini's theorem we get, by the same
arguments, that
\begin{equation*}
\langle \widetilde{\pi}_\mu^v,\widetilde{\pi}_\delta^w\rangle_t =\delta_{\mu
, \delta}\frac{e^{2t\langle \mu +2\rho ,\mu\rangle}}{d(\mu )}\langle
u,w\rangle_\mu
\end{equation*}
and the Lemma follows.
\end{proof}

\begin{theorem}
\label{th-holext} Let $s,R,S>0$. Assume that $\mu \in v(\mu )\in V_\mu $ is such that
$\|v_\mu\|_\mu \le Re^{S\|\mu\|}$. Then
\begin{equation*}
F(z)=\sum_{\mu\in\Lambda^+}d(\mu )e^{-s\langle \mu +2\rho ,\mu \rangle}%
\widetilde{\pi}_\mu^{v (\mu )} (z)
\end{equation*}
defines a holomorphic function on $M_\mathbb{C}$. If $L\subset M_\mathbb{C}$ is
compact, then there exists $C (L)>0$, independent of $(v_\mu)_\mu$, such
that
\begin{equation}  \label{le-estimate}
|F(z)|\le C (L)R\, .
\end{equation}
\end{theorem}

\begin{proof}
Let $L\subseteq M_\mathbb{C}$ be compact. Let $F_\mu (z)=d(\mu )\langle
v (\mu ),\pi_\mu (\bar z )e_\mu\rangle_\mu$. Then $F_\mu$ is holomorphic and
\begin{equation*}
|F_\mu (z)|\le e^{(S+C_L)\|\mu\|}\|u\|_\mu \le Re^ {(S+C_L)\|\mu\|}
\end{equation*}
by Lemma \ref{le-PropWidetildePi}. As, $\mu \mapsto
d(\mu )$ is a polynomial function of degree $\sum_{\alpha \in \Sigma^+} \dim_\C \fu_{\C \alpha}$ and
$\langle \mu +2\rho ,\mu \rangle\ge 0$, it follows that the function $\Lambda^+ \to \mathbb{R}^+$,
\begin{equation*}
\mu \mapsto d(\mu )(1+\|\mu \|^2)^{k} e^{(S+C_L)\|\mu \|}e^{-s (\mu^2
+2\rho \cdot \mu )}
\end{equation*}
is bounded for each $k\in\mathbb{Z}^+$. Hence, there exists a constant $c(k,L,s)$ independent of $\mu$ such that
\begin{equation*}
d(\mu ) e^{(S+C_L)\|\mu \|}e^{-s ( \mu^2 +2\rho \cdot \mu)}\le
c(k,L,s)(1+\|\mu \|^2)^{-k}
\end{equation*}
for all $\mu \in \Lambda^+$. By Lemma 1.3 in \cite{Sugiura} (see also in Lemma 5.6.7 in
\cite{Wallach}) there exists $k_0\in \mathbb{N}$ such that
$\sum_{\mu \in\Lambda^+} (1+\|\mu \|^2)^{-k_0}$ converges. Hence
\begin{equation}  \label{estimate}
\sum_{\mu\in \Lambda^+} d (\mu )e^{-s\langle \mu +\rho, \rho\rangle}
|F_\mu(z)| \le c(k_0,L,s)R\sum_{\mu\in\Lambda^+}(1+\|\mu \|^2)^{-k_0}\le
C(L)R
\end{equation}
converges uniformly on $L$, and hence defines a holomorphic function on $M_%
\mathbb{C}$. The estimate (\ref{le-estimate}) is (\ref{estimate}).
\end{proof}

For $z=a\cdot o,w=b\cdot o\in M_\mathbb{C}$, write
\begin{equation*}
\widetilde{\psi}_\mu (w^*z):=\widetilde{\psi}_\mu (\sigma(b)^{-1}a) =L(\overline{w}%
)\widetilde{\psi}_\mu (z)\, .
\end{equation*}
Then $z,w\mapsto \widetilde{\psi}_\mu (w^*z)$ is well defined, holomorphic
in $z$ and anti-holomorphic in $w$. The above Lemma implies that
\begin{equation*}
(z,w)\mapsto \sum_{\mu\in\Lambda^+} d(\mu) e^{-2t\langle \mu + 2\rho
,\mu\rangle}\widetilde{\psi}_\mu ( w^*z)= \sum_{\mu\in\Lambda^+} d(\mu)
e^{-2t\langle \mu + 2\rho ,\mu\rangle}\langle \pi_\mu (\overline w)e_\mu, \pi_\mu
(\overline z) e_\mu\rangle
\end{equation*}
defines a function on $M_\mathbb{C}\times M_\mathbb{C}$, holomorphic in $z$
and anti-holomorphic in $w$.

Denote the unitary representation of $U$ on $\mathcal{H}_t(M_\mathbb{C})$ by
$\tau_t$. Then we have the following theorem:

\begin{theorem}
\label{th-PlancherelForHt} Let $t>0$ then $(\tau_t,\mathcal{H}_t(M_\mathbb{C}%
))\simeq_U (L,L^2(M))$. Furthermore, the reproducing kernel for $\mathcal{H}%
_t(M_\mathbb{C})$ is given by
\begin{equation*}
K_t(z,w)=\sum_{\mu\in\Lambda^+} d(\mu) e^{-2t\langle \mu + 2\rho ,\mu\rangle}%
\widetilde{\psi}_\mu ( w^*z)\, .
\end{equation*}
\end{theorem}

\begin{proof}
We use the \textit{Restriction Principle} introduced in \cite{Olafsson} (see
also \cite{O1}) and Lemma \ref{focknorm} for the proof. Define $R : \mathcal{H}%
_t(M_\mathbb{C})\to C^\infty (M)$ by $RF:=F|_M$. Then $R$ commutes with the
action of $U$. As $M$ is totally real in $M_\mathbb{C}$ it follows that $R$
is injective. As $M$ is compact, we get from Theorem \ref{th-propHt} that
there exists a constant $C_M>0$ such that $\sup_{x\in M}|RF(x)|\le C_M
\|F\|_t$. Thus
\begin{equation*}
\|RF\|_2\le C_M\|F\|_t\, .
\end{equation*}
Thus $R: \mathcal{H}_t(M_\mathbb{C})\to L^2(M)$ is a continuous $U$%
-intertwining operator.

By Lemma \ref{focknorm}, we have $L^2(M)_\mu \subset \mathrm{Im}(R)$ for each
$\mu\in \Lambda^+$. Thus $\mathrm{Im}(R)$ is dense in $L^2(M)$. Write $R^*=U_t\sqrt{RR^*}$.
Then $U_t:L^2(M)\to \mathcal{H}_t(M_\mathbb{C})$ is an unitary
isomorphism and $U$-intertwining operator.

For $\mu\in\Lambda^+$ let
\begin{equation*}
\mathcal{H}_t(M_\mathbb{C})_\mu=\{\widetilde{\pi}^u_\mu\mid u\in V_\mu
\}=U_t(L^2(M)_\mu)\, .
\end{equation*}
Then $\mathcal{H}_t(M)=\bigoplus_{\mu\in\Lambda^+}\mathcal{H}_t(M_\mathbb{C}%
)_\mu$. As $\mathcal{H}_t(M_\mathbb{C})_\mu\perp \mathcal{H}_t(M_\mathbb{C}%
)_\delta$ for $\mu\not=\delta$ it follows that $K=\sum_{\mu\in\Lambda^+}K_\mu
$ where $K_\mu$ is the reproducing kernel for $\mathcal{H}_t(M_\mathbb{C}%
)_\mu$. Now, note that
\begin{equation*}
L(\overline{w}) \widetilde{\psi}_\mu = \widetilde{\pi}_\mu^{\pi_\mu (\overline
w)e_\mu}(z) \, .
\end{equation*}
Thus Lemma \ref{focknorm} implies that
\begin{equation*}
d(\mu )e^{-2t\langle \mu +2\rho,\mu\rangle} \langle \widetilde{\pi}^u_\mu ,
L(\overline{w})\widetilde{\psi}_\mu\rangle_t = \langle u,\pi_\mu (\overline
w)e_\mu\rangle_\mu =\widetilde{\pi}_\mu^u (z)\, .
\end{equation*}
Hence $K_\mu (z,w)=d(\mu )e^{-2t\langle \mu + 2\rho ,\mu\rangle}\widetilde{%
\psi}_\mu ( w^*z)$ and the theorem follows.
\end{proof}

\begin{theorem}
Define $k_t : M_\mathbb{C} \to \mathbb{C}$ by
\begin{equation*}
k_t(z):= K_t(z,o)=\sum_{\mu\in\Lambda^+}d(\mu )e^{-2t\langle \mu +2\rho
,\mu\rangle}\widetilde{\psi}_\mu (z)\, .
\end{equation*}
Let $a,b\in U_\mathbb{C}$ and $z,w\in M_\mathbb{C}$. Then

\begin{enumerate}
\item $K_t(a\cdot z,b\cdot w)=K_t(b^*a\cdot z,w)$.

\item $K_t(z,w)=k_t(w^*z)$.

\item $k_t\in\mathcal{H}_t(M_\mathbb{C})^{K_\mathbb{C}}$.

\item $\overline{k_t(z)}=k_t(z^*)$.

\item $k_t|_M$ is real-valued.

\item $k_t(x\cdot o)=k_t(x^{-1}\cdot 0)$ for all $x\in U_\mathbb{C}$.
\end{enumerate}
\end{theorem}

\begin{proof}
Everything except (5) and (6) follows from Theorem \ref{th-PlancherelForHt}.
Let $w^*\in W$ be the unique element such that $w^*(\Sigma^+)=-\Sigma^+$.
Then $-w^*\Lambda^+=\Lambda^+$ and if $\mu\in \Lambda^+$ then
\begin{equation}  \label{eq-psiw*}
\psi_{-w^*\mu}(x)=\psi_\mu (x^{-1})=\overline{\psi_\mu (x)}\, ,\quad x\in U.
\end{equation}
This is well known and follows easily from Lemma \ref{le-phipsi}: We have
\begin{eqnarray*}
\psi_{-w^*\mu}(x)&=&\widetilde{\varphi}_{i(-w^*\mu +\rho)}(x) \\
&=&\widetilde{\varphi}_{-w^*(i(\mu +\rho))}(x) \\
&=&\widetilde{\varphi}_{-i(\mu +\rho)}(x) \\
&=&\widetilde{\varphi}_{i(\mu+\rho)}(x^{-1}) \\
&=&\psi_\mu(x^{-1}) \\
&=&\overline{\psi_\mu (x)}
\end{eqnarray*}
It follows that $k_t(x)=\frac{1}{2} \sum_{\mu\in\Lambda^+}(\psi_\mu
(x)+\psi_{-w^{*}\mu}(x)) =\sum_{\mu\in\Lambda^+}\mathrm{Re}(\psi_\mu (x))$ and
hence $k_t(x)$ is real.

By the same argument, we see that for $x\in U$ we have
\begin{equation*}
k_t(x)=\frac{1}{2}\sum_{\mu\in\Lambda^+}(\psi_\mu (x)+\psi_{-w^*\mu}(x)) =%
\frac{1}{2}\sum_{\mu\in\Lambda^+} (\psi_\mu (x)+\psi_\mu (x^{-1}))
\end{equation*}
so $k_t(x)=k_t(x^{-1})$ on $U$. But both sides are holomorphic on $U_\mathbb{%
C}$ and therefore agree on $U_\mathbb{C}$.
\end{proof}

Define
\begin{equation*}
\mathcal{F}_t(\Lambda^+):=\left\{a :\Lambda^+\to\prod_{\mu\in\Lambda^+}
V_\mu\,\left|\,  a(\mu)\in V_\mu\text{ and }
\sum_{\mu\in \Lambda^+} d(\mu )e^{2t\langle \mu+2\rho ,\mu\rangle} \|a(\mu
)\|^2_\mu<\infty\right.\right\}\, .
\end{equation*}
Then $\mathcal{F}_t (\Lambda^+)$ is a Hilbert space with inner product
\[\langle
a, b\rangle_{\mathcal{F}} = \sum_{\mu\in \Lambda^+} d(\mu )e^{2t\langle
\mu+2\rho ,\mu\rangle} \langle a(\mu ),b(\mu )\rangle_\mu\]
and $U$ acts
unitarily on $\mathcal{F}_t(\Lambda^+)$ by
\begin{equation*}
[\sigma_{t}(x)(a)](\mu) = \pi_\mu (x)a(\mu )\, .
\end{equation*}

\begin{theorem}
\label{th-SeqSp}
We have the followings.
\begin{enumerate}
\item For $\mu\in\Lambda^+$ let $v^1_\mu,\ldots ,v^{d(\mu )}_\mu$ be an
orthonormal basis for $V_\mu$. Let $\widetilde{\pi}_\mu^j=\widetilde{\pi}%
^{v^j_\mu}_\mu$. Then $\displaystyle{\left\{\left. \sqrt{d(\mu )}
e^{-t\langle \mu+2\rho ,\mu\rangle}\, \widetilde{\pi}^{j}_\mu \, \right|\,
\mu\in\Lambda^+\, ,\,\, i=1,\ldots , d(\mu)\right\}}$ is an orthonormal
basis for $\mathcal{H}_t(M_\mathbb{C})$.

\item The map $a\mapsto \sum_{\mu\in\Lambda^+}d(\mu) \widetilde{\pi}^{a(\mu)}_\mu$ is a unitary $U$-isomorphism, $\mathcal{F}_t(\Lambda^+)\simeq
\mathcal{H}_t(M_\mathbb{C})$.

\item The set $\{\sqrt{d(\mu )}e^{-t\langle \mu+2\rho ,\mu\rangle}\,
\widetilde{\psi}_\mu\mid \mu \in \Lambda^+\}$ is an orthonormal basis for $%
\mathcal{H}_t(M_\mathbb{C})^K$.

\item $\mathcal{H}_t(M_\mathbb{C})^K$ is isometrically isomorphic to the
sequence space
\begin{equation*}
\left\{ a: \Lambda^+ \to \mathbb{C} \, \left|\, \sum_{\mu\in\Lambda^+} d(\mu )e^{2t\langle
\mu +2\rho, \mu\rangle}|a(\mu)|^2 < \infty \right.\right\}\, .
\end{equation*}
The isomorphism is given by $a\mapsto \sum_{\mu\in\Lambda^+}d(\mu) a(\mu )%
\widetilde{\psi}_\mu$.
\end{enumerate}
\end{theorem}

\begin{proof}
(1) This follows from Theorem \ref{th-PlancherelForHt} and Lemma \ref%
{focknorm}.

(2) For $a\in\mathcal{F}_t(\Lambda^+ )$ define $F :=
\sum_{\mu\in\Lambda^+}d(\mu )\widetilde\pi^{a(\mu )}_\mu$. Let $%
v_\mu=e^{t\langle \mu+2\rho,\mu\rangle}a(\mu )$. Then the sequence $%
\{\|v_\mu\|_\mu\}$ is bounded. As $F=\sum_{\mu \in\Lambda^+} d(\mu
)e^{-t\langle \mu +2\rho,\mu\rangle }\widetilde \pi_\mu^{v_\mu}$ it follows
from Lemma \ref{th-holext} that the series converges and that $F$ is
holomorphic. Furthermore,
\begin{eqnarray*}
\|F\|_t^2&=&\sum_{\mu\in\Lambda^+} d(\mu)^2\|\widetilde\pi_\mu^{a(\mu
)}\|_t^2 \\
&=&\sum_{\mu\in\Lambda^+} d(\mu)^2 e^{2t\langle
\mu+2\rho,\mu\rangle}d(\mu)^{-1}\|a (\mu )\|^2_\mu \\
&=&\sum_{\mu\in\Lambda^+} d(\mu) e^{2t\langle \mu+2\rho,\mu\rangle}\|a (\mu
)\|^2_\mu<\infty\, .
\end{eqnarray*}
Hence $F\in\mathcal{H}_t(M_\mathbb{C})$ and $a\mapsto F$ is an isometry.
Now, let $F\in\mathcal{H}_t(M_\mathbb{C})$. By Theorem \ref%
{th-PlancherelForHt} $F=\sum_{\mu }d (\mu )e^{-2t\langle \mu +2\rho, \mu
\rangle }\langle F,\widetilde\pi^j_\mu\rangle_t \widetilde\pi^i_\mu$ and $%
\|F\|^2=\sum_{\mu\in\Lambda^+} d(\mu )e^{-2t\langle \mu +2\rho, \mu \rangle
}|\langle F,\widetilde\pi^j_\mu\rangle_t|^2<\infty$. Letting
\begin{equation*}
a(\mu )=\sum_{j=1}^{d(\mu )} e^{-2t\langle \mu +2\rho ,\mu\rangle }\langle
F,\widetilde\pi^j_\mu\rangle_t\, v^j_\mu
\end{equation*}
we get
\begin{equation*}
\sum_{\mu\in\Lambda^+} d (\mu )e^{2t\langle \mu +2\rho , \mu\rangle}\|a (\mu
)\|^2= \|F\|_t^2<\infty
\end{equation*}
and $F = \sum_{\mu\in\Lambda^+}d(\mu )\widetilde\pi^{a(\mu )}_\mu$. Hence $%
a\mapsto F$ is a unitary isomorphism. That this map is an intertwining
operator follows from the equation
\begin{equation*}
\widetilde{\pi}_\mu^v(x^{-1}y)= \widetilde{\pi}_\mu^{\mu_\mu (x)v}(y)\, .
\end{equation*}

(3) and (4) now follows as $\widetilde{\psi}_\mu =\widetilde{\pi}^{e_\mu}_\mu
$.
\end{proof}

\section{Segal-Bargmann Transforms on $L^2(M)$ and $L^2(M)^K$}\label{s4}

\noindent In this section, we introduce the heat equation and the heat semigroup $%
e^{t\Delta}$. We show that if $f\in L^2(M)$ then $e^{t\Delta}f$ extends to the holomorphic
function $H_tf$ on $M_\mathbb{C}$ and $H_tf\in\mathcal{H}_t(M_\mathbb{C})$. Then we
show that the map $H_t:L^2(M) \to \mathcal{H}_t(M_\mathbb{C})$, $H_t(f) = H_tf$, is the unitary
isomorphism $U_t$ in the proof of Theorem \ref{th-PlancherelForHt}. The
isomorphism $H_t :L^2(M)\to \mathcal{H}_t (M_\mathbb{C})$ was first
established in \cite{Hall94} and \cite{Stenzel}. A different proof was later
given by Faraut in \cite{Faraut1} using Gutzmer's formula. In \cite{HZ09} the Restriction Principle was used. Our proof is also based on the Restriction Principle and uses some ideas from \cite{Faraut1}. The $K$-invariant case was treated in Chapter 4 of \cite{Wiboonton}.

The heat equation on $M$ is the Cauchy problem
\begin{align*}
\Delta u(x,t) &= \frac{\partial u}{\partial t}(x,t)\, , \quad (x,t) \in M
\times(0,\infty) \\
\lim_{t\rightarrow 0^+}u(x,t) &= f(x)\, ,\quad f \in L^2(M)\text{ (the
initial condition)}  \notag
\end{align*}
where $\Delta$ is the Laplace-Beltrami operator on $M$ defined by $\langle
\cdot, \cdot \rangle$. $\Delta$ is a self adjoint negative operator on $M$
and a solution to the heat equation is give by the \textit{heat semigroup} $e^{t\Delta}$ applied to $f$, $u(\cdot ,t)= e^{t\Delta}f$.

\begin{lemma}
\label{eigenvalues} $\Delta$ acts on $L^2(M)_\mu$ by $-\langle \mu +2\rho,
\mu \rangle\mathrm{Id} $.
\end{lemma}

\begin{proof}
This is well know for the Laplace-Beltrami operator constructed by the
Killing form metric. But scaling the inner product on $\mathfrak{s}$ by a
constant $c>0$, results in scaling the Laplace-Beltrami operator as well as
the inner product on $\mathfrak{s}^*$ by $1/c$.
\end{proof}

\begin{lemma}
\label{le-Heatonf} Let $f\in L^2(M)$. Write $f=\sum_{\mu\in\Lambda^+} f_\mu$
with $f_\mu = d(\mu ) f* \widetilde{\psi}_\mu \in L^2(M)_\mu\subset C^\infty
(M)$. Then
\begin{equation}  \label{eq-etDelta}
e^{t\Delta}f=\sum_{\mu\in\Lambda^+}e^{-t\langle \mu +2\rho
,\mu\rangle}f_\mu\, .
\end{equation}
\end{lemma}

\begin{proof}
This follows from Lemma \ref{eigenvalues}.
\end{proof}

We call the map $H_t : L^2(M)\to \mathcal{O} (M_\mathbb{C})$ the \textit{Heat
transform} or the \textit{Segal-Bargmann} transform on $L^2(M)$.

\begin{theorem}
\label{eq-HtIso} If $f\in L^2(M)$ then $H_t(f)\in \mathcal{H}_t (M_\mathbb{C}%
)$ and $H_t:L^2(M)\to \mathcal{H}_t(M_\mathbb{C}) $ is a unitary
$U$-isomorphism. Furthermore

\begin{enumerate}
\item Let $h_t=k_{t/2}$. Then $H_tf (z)= (f*h_t)(z)$, see (\ref{def-convolution}).

\item $H_t=U_t$ where $U_t:L^2(M)\to \mathcal{H}_t(M_\mathbb{C})$ is the
unitary isomorphism from the proof of Theorem \ref{th-PlancherelForHt}.
\end{enumerate}
\end{theorem}

\begin{proof}
We have
\begin{equation*}
H_tf=\sum_{\mu\in\Lambda^+} d(\mu )e^{-t\langle \mu +2\rho,\mu\rangle }
\widetilde \pi^{\widehat{f}_\mu}_\mu=\sum_{\mu\in\Lambda^+}d(\mu)%
\widetilde{\pi}_\mu^{a(\mu )}
\end{equation*}
with $a(\mu )=e^{-t\langle \mu +2\rho,\mu\rangle }\widehat{f}_\mu$. By
Theorem \ref{th-Plancherel1}
\begin{equation*}
\sum_{\mu\in\Lambda^+} d(\mu )e^{2t\langle \mu+2\rho ,\mu\rangle}\|a(\mu
)\|_\mu^2 =\sum_{\mu\in\Lambda^+}d(\mu )\|\widehat{f}_\mu\|_\mu^2\newline
=\|f\|_2^2\, .
\end{equation*}
By Theorem \ref{th-SeqSp}, $H_tf$ extends to a holomorphic function on $M_%
\mathbb{C}$, $H_tf\in\mathcal{H}_t (K_\mathbb{C})$, and $\|H_tf\|_t = \|f\|_2$%
. That $H_t$ is bijective follows easily not only from the representation theory but
also from the fact, which we will prove in a moment, that $H_t = U_t$.

(1) For $f\in L^1(M)$ and $g\in L^1(M)^K$, or $g$ holomorphic and $K$%
-invariant on $M_\mathbb{C}$, define
\begin{equation}  \label{def-convolution}
(f*g)(m):=\int_U f(x\cdot o)g(x^{-1}\cdot m)\, dx\, .
\end{equation}
We have $|f|*|\widetilde{\psi}_\mu| (z)\le e^{C_U\|\mu \|}\|f\|_2$. Hence $%
\sum_{\mu\in\Lambda^+} d (\mu )e^{-t\langle \mu +\rho, \mu\rangle }|f|*|%
\widetilde{\psi}_\mu| (z)<\infty$ and we can interchange the integration and
summation to get (with some obvious abuse of notation)
\begin{equation*}
f*h_t=\sum_{\mu\in\Lambda^+} d (\mu )e^{-t\langle \mu +2\rho, \mu\rangle }f*%
\widetilde{\psi}_\mu = \sum_{\mu\in\Lambda^+} e^{-t\langle \mu +2\rho
,\mu\rangle}f_\mu =H_tf
\end{equation*}
where the last equality follows from Lemma \ref{le-Heatonf}.

(2) We use the ideas from \cite{Olafsson}. Let $R$ and $U_t$ be as in the
proof of Theorem \ref{th-PlancherelForHt}. Let $f\in L^2(M)$. Then
\begin{equation*}
R^*f(z)=\langle R^*f,K_t(\,\cdot\, ,z)\rangle_2 =\langle f,RK_t(\,\cdot \,
,z)\rangle_2 =\int_M f(x)K_t(z,x)\, dx =\int_M f(x)k_t (x^{-1}z)\, dx\, .
\end{equation*}
In particular
\begin{equation*}
RR^*f(m)=f*k_t (m)\, .
\end{equation*}
It follows that $RR^*f=e^{2t\Delta} f$. As $s\mapsto e^{s\Delta}$ is an
operator valued semigroup it follows that $\sqrt{RR^*} =e^{t\Delta}$. The
image of $H_t=\sqrt{RR^*}$ is dense in $L^2(M)$. Let $g=H_tf\in H_t(L^2(M))$.
Then
\begin{equation*}
RU_tg = RR^*f=H_tg\, .
\end{equation*}
As $U_tg$ and $H_tg$ are holomorphic and agree on $M$ it follows that $U_tg
(z)=H_tg(z)$ on $M_\mathbb{C}$. The image of $H_t$ is dense in $L^2(M)$ and
both $U_t$ and $H_t$ are continuous, hence $U_t=H_t$.
\end{proof}

Define $F_t: \bigoplus_{\mu \in \Lambda^+,\, d}V_\mu\to \mathcal{F}%
_t(\Lambda^+) $ by $F_t(v)(\mu ):=e^{-t\langle \mu +2\rho ,\mu \rangle}
v(\mu )$.

\begin{corollary}
\label{co-HeatSeq} We have a commutative diagram of unitary $U$-isomorphisms
\begin{equation*}
\xymatrix{L^2(M) \ar[r]^{H_t} \ar[d]_{\widehat{\hbox to0.5em{}}} &
\cH_t(M_\C)\ar[d]^{F\mapsto a}\\
\IFM \ar[r]_{F_t }&\cF_t(\LK )  }\, .
\end{equation*}
\end{corollary}

\begin{proof}
This follows from (\ref{eq-etDelta}), Theorem \ref{th-Plancherel1}, Theorem %
\ref{th-SeqSp}, and Theorem \ref{eq-HtIso}.
\end{proof}

\section{Propagations of Compact Symmetric Spaces}\label{s5}

\noindent In this section, we describe the results, which we need later on, from \cite%
{OlafssonWolf,JWCompact} on limits of symmetric spaces and the related
representation theory and harmonic analysis. We mostly follow the
discussion and notations in \cite{OlafssonWolf}. Most of the material on
spherical representations is taking from Section 6 in \cite{OlafssonWolf, OW2}. We keep
the notations from the previous sections and indicate the dependence on the
symmetric spaces by the indices $m, n$ etc. In particular $M_n = U_n/K_n$, $n\in
\mathbb{N}$ is a sequence of simply connected symmetric spaces of compact
type. We assume that for $m\ge n$, $U_n \subseteq U_m$ and $\theta_m|_{%
\mathfrak{u}_n} = \theta_n$. Then $K_n=K_m\cap U_n$, $\mathfrak{k}_m\cap
\mathfrak{u}_n=\mathfrak{k}_n$, and $\mathfrak{s}_m\cap \mathfrak{u}_n$. We
recursively choose maximal commutative subspaces $\mathfrak{a}_m \subset
\mathfrak{s}_m$ such that $\mathfrak{a}_n =\mathfrak{a}_m\cap \mathfrak{u}_n$
for all $m \geq n$. Let $r_{n} = \dim \mathfrak{a}_{n}$ be the rank of $M_{n}
$.

As before, we let $\Sigma_{n} = \Sigma ( \mathfrak{u}_{n,\mathbb{C}},
\mathfrak{a}_{n,\mathbb{C}} )$ denote the system of restricted roots of $
\mathfrak{a}_{n,\mathbb{C}}$ in $\mathfrak{u}_{n,\mathbb{C}}$. We can--and
will--choose positive systems so that $\Sigma_n^+\subseteq \sigma_m^+|_{\mathfrak{a}_n}$. Let
\begin{equation*}
\Sigma_{1/2,n} = \left\{\alpha \in \Sigma_{n}\mid \frac{1}{2}\alpha \notin
\Sigma_{n}\right\} \text{ and } \Sigma_{2,n} = \left\{\alpha \in
\Sigma_{n}\mid 2\alpha \notin \Sigma_{n}\right\}\, .
\end{equation*}

Then $\Sigma_{1/2,n}$ and $\Sigma_{2,n}$ are reduced root systems (see Lemma
3.2, p. 456 in \cite{He78}). Consider the positive systems $%
\Sigma_{1/2,n}^{+} := \Sigma_{1/2,n} \cap \Sigma_{n}^{+}$ and $%
\Sigma_{2,n}^{+} := \Sigma_{2,n} \cap \Sigma_{n}^{+}$. Let
$
\Psi_{1/2,n}$ and $\Psi_{2,n}$
denote the sets of simple roots for $\Sigma_{1/2,n}^{+}$ and $%
\Sigma_{2,n}^{+}$ respectively.

Suppose for a moment that $M_n$ is an irreducible symmetric space for every $n$. We say that $M_m$ \textit{propagates} $M_n$ if $\Sigma_{1/2,n}=\Sigma_{1/2,m}$ or
we only add simple roots to the left end of the Dynkin
diagram for $\Psi_{1/2,n}$ to obtain the Dynkin diagram for $\Psi_{1/2,m}$.
In particular, $\Psi_{1/2,n}$ and $\Psi_{1/2,m}$ are of the same type. In
general, if $M_m\simeq M_m^1\times \ldots \times M_m^r$ and $M_n \simeq
M_n^1\times M_n^s$ with $M_m^i$ and $M_n^j$ irreducible, then $M_m$ \textit{%
propagates} $M_n$ if we can enumerate the irreducible factors $M_m^i$ and $%
M_n^j$ such that $M_m^i$ propagates $M_n^i$ for $i = 1,2,...,s$. We refer to
the discussion in Section 1 of \cite{OW2} for more details.

From now on, we assume that $M_m$ \textit{propagates} $M_n$ for all $m \geq n
$. We call the sequence $\{M_n = U_n/K_n\}$, the \textit{propagating sequence%
} of symmetric spaces of compact type. This includes sequences of symmetric
spaces from each line of the following table of classical symmetric spaces,
see \cite{He84}, Table V, page 518:

\begin{equation}  \label{symmetric-case-class}
\begin{tabular}{|c|l|l|c|c|}
\hline
\multicolumn{5}{|c|}{} \\
\multicolumn{5}{|c|}{Compact Irreducible Riemannian Symmetric $M = U/K$} \\
\multicolumn{5}{|c|}{} \\ \hline\hline
Type & \multicolumn{1}{c}{$U$} & \multicolumn{1}{|c}{$K$} &
\multicolumn{1}{|c}{rank$\, M$} & \multicolumn{1}{|c|}{dim$\, M$} \\
\hline\hline
$A_{n-1}$ & $\mathrm{SU} (n)\times \mathrm{SU} (n)$ & $\mathrm{diag\, }
\mathrm{SU} (n)$ & $n-1$ & $n^2-1$ \\ \hline
$B_{n}$ & $\mathrm{SO} (2n+1)\times \mathrm{SO} (2n+1)$ & $\mathrm{diag\, }
\mathrm{SO} (2n+1)$ & $n$ & $2n^2+n$ \\ \hline
$C_n$ & $\mathrm{Sp} (n)\times \mathrm{Sp} (n)$ & $\mathrm{diag\, } \mathrm{%
Sp} (n)$ & $n$ & $2n^2+n$ \\ \hline
$D_n$ & $\mathrm{SO} (2n)\times \mathrm{SO} (2n)$ & $\mathrm{diag\, }
\mathrm{SO} (2n)$ & $n$ & $2n^2-n$ \\ \hline
$AI$ & $\mathrm{SU} (n)$ & $\mathrm{SO} (n)$ & $n-1$ & $\frac{(n-1)(n+2)}{2}$
\\ \hline
$AII$ & $\mathrm{SU} (2n)$ & $\mathrm{Sp} (n) $ & $n-1$ & $2n^2-n-1$ \\
\hline
$AIII$ & $\mathrm{SU} (p+q)$ & $\mathrm{S} (\mathrm{U} (p)\times \mathrm{U}%
(q))$ & $\min(p,q)$ & $2pq$ \\ \hline
$BDI$ & $\mathrm{SO} (p+q)$ & $\mathrm{SO} (p) \times \mathrm{SO} (q)$ & $%
\min(p,q)$ & $pq$ \\ \hline
$DIII$ & $\mathrm{SO} (2n)$ & $\mathrm{U} (n)$ & $[\frac{n}{2}]$ & $n(n-1)$
\\ \hline
$CI$ & $\mathrm{Sp} (n)$ & $\mathrm{U} (n)$ & $n$ & $n(n+1)$ \\ \hline
$CII$ & $\mathrm{Sp} (p+q)$ & $\mathrm{Sp} (p) \times \mathrm{Sp} (q)$ & $%
\min(p,q)$ & $4pq$ \\ \hline
\end{tabular}%
\end{equation}
But we can also include an inclusion like
\begin{equation*}
\mathnormal{SU}(n)/\mathnormal{SO}(n) \subset (\mathnormal{SU}(n) \times
\mathnormal{SU}(n))/\mathrm{diag}(\mathnormal{SU}(n) \times \mathnormal{%
SU}(n)).
\end{equation*}

Now, let $\Psi_{2,n} = \{\alpha_{n,1} , \ldots , \alpha_{n,r_{n}}\}$.
According to the root systems discussed in Section 2
of \cite{OlafssonWolf}, we can choose the ordering so that for $j \leq r_{n}$,
$\alpha_{m,j}$ is the unique element of $\Psi_{2,m}$ whose restriction to $%
\mathfrak{a}_{n}$ is $\alpha_{n,j}$. Define $\xi_{n,j}\in i\mathfrak{a}_n^*$
by
\begin{equation}  \label{def-xiNj}
\frac{\langle \xi_{n,j},\alpha_{n,i}\rangle}{\langle
\alpha_{n,i},\alpha_{n,i}\rangle} =\delta_{i,j}
\end{equation}

The weights $\xi_{n,j}$ are the \textit{class-1 fundamental weights} for $(
\mathfrak{u}_{n} , \mathfrak{k}_{n} )$. We set
\begin{equation*}
\Xi_{n} = \left\{\xi_{1,r_{1}},\ldots,\xi_{n,r_{n}}\right\}.
\end{equation*}
It is clear by the definition of $\Lambda^+_n$ that
\begin{equation}  \label{eq-Lambdan}
\Lambda^+_n=\sum_{j=1}^{r_n}\mathbb{Z}^+ \xi_{n,j}\, .
\end{equation}

\begin{lemma}[\protect\cite{Wolf2007}, Lemma 6, \protect\cite{OlafssonWolf},
Lemma 6.7]
\label{Wolflemma} Recall the root ordering of $(5.2)$. If $1\leq j\leq r_{n}$
then $\xi_{m,j}$ is the unique element of $\Xi_{m}$ whose restriction of $%
\mathfrak{a}_{n}$ is $\xi_{n,j}$.
\end{lemma}

This allows us to construct the map $\iota_{n,m} :\Lambda^+_n\to
\Lambda^+_m$
\begin{equation}
\iota_{n,m}\left(\sum_{j=1}^{r_n} k_j\xi_{n,j}\right):=\sum_{j=1}^{r_n}k_j\xi_{m,j}\, .
\end{equation}
Note that $\iota_{n,m}(\mu )|_{\mathfrak{a}_n}=\mu$ and if $\delta\in
\Lambda^+_m$ is such that $\delta = \sum_{j=1}^{r_n} k_j\xi_{m,j}$, then $%
\delta|_{\mathfrak{a}_n}\in \Lambda^+_n$ and $\iota_{n,m}(\delta|_{\mathfrak{%
a}_n})=\delta$.

\begin{lemma}[\protect\cite{OlafssonWolf}, Lemma 6.8]
\label{resmult} Assume that $\delta \in \Lambda^+_m$ is a combination of the
first $r_n$ fundamental weights, $\delta = \sum_{j=1}^{r_n} k_j \xi_{m,j}$.
Let $\mu:=\delta|_{\mathfrak{a}_n}=\sum_{j=1}^{r_n}k_j \xi_{n,j}\, $. If $%
v_\delta$ is a nonzero highest weight vector in $V_{\delta}$ then $%
\langle\pi_{\delta} (U_n)v_\delta\rangle$, the linear span of $\{\pi_\mu
(g)v_\delta\mid g\in U_n\}$, is an irreducible representation of $U_n$ which
is isomorphic to $(\pi_{\delta},V_{\mu})$. Furthermore, $v$ is a highest
weight vector for $\pi_\mu$ and $\pi_{\mu}$ occurs with multiplicity one in $%
\pi_{\delta}|_{G_n}$.
\end{lemma}

The point of this discussion is, that if $\mu\in\Lambda^+_n$, then we
can--and will--view $V_\mu$ as a subspace of $V_{\iota_{n,m}(\mu )}$ such
that $\langle u,v\rangle_\mu=\langle u,v\rangle_{\iota_{n,m}(\mu )}$ for all
$u,v\in V_\mu$.

We note that the $K_n$-invariant vector $e_\mu\in V_\mu$ is not necessarily $%
K_m$-invariant. But the projection of $e_{\iota_{n,m}(\mu )}$ onto $V_\mu$
is always non-zero and $K_n$ invariant. This follows from the Lemma \ref%
{resmult} and the fact that $\langle v_\delta ,e_\delta\rangle_\delta\not=0$
(see \cite{He84}, the proof of Theorem 4.1, Chapter V). In particular $%
e_{\iota_{n,m}(\mu )}=ce_\mu +f_{n,m;\mu}$ for some $K_n$-fixed vector $%
f_{n,m;\mu}$, orthogonal to $e_\mu$.

\section{The Segal-Bargman Transform on the Direct Limit of $\{L^2(M_n)\}_n$}\label{s6}

\noindent In this section, we recall the isometric $U_n$-embedding on $%
L^2(M_n)$ into $L^2(M_m)$ due to J. Wolf, \cite{JWCompact}. We then discuss
similar construction for the Fock spaces and show that the Segal-Bargmann
transform extends to the Hilbert-space direct limit.

First, define $\gamma_{m,n} : L^2(M_n)\longrightarrow L^2(M_m)$ by
\begin{eqnarray*}
\gamma_{n,m}(f) &:=& \sum_{\mu\in\Lambda^+_n} d((\iota_{n,m}(\mu ))\sqrt{%
\frac{d(\mu )}{d(\iota_{n,m}(\mu ))}}\, \langle \widehat{f}_\mu
,\pi_{\iota_{n,m}(\mu )}(\, \cdot \, )e_{\iota_{n,m}(\mu
)}\rangle_{\iota_{n,m}(\mu )} \\
&=&\sum_{\mu\in\Lambda^+_n} \sqrt{d((\iota_{n,m}(\mu ))d(\mu )}\, \langle
\widehat{f}_\mu,\pi_{\iota_{n,m}(\mu )}(\, \cdot \, )e_{\iota_{n,m}(\mu
)}\rangle_{\iota_{n,m}(\mu )} \, .
\end{eqnarray*}

Clearly each $\gamma_{m,n}$ is linear. Then by Theorem \ref{th-Plancherel1},
in particular (\ref{eq-FourierSeries}) it follows that $\gamma_{n,m}$ is an
isometry. (\ref{eq-intertwining}) implies that $\gamma_{n,m}$ is an
intertwining operator. Moreover, if $n \leq m \leq p$, then
\begin{equation*}
\gamma_{n,p} = \gamma_{m,p} \circ \gamma_{n,m}.
\end{equation*}
Therefore we have a direct system of Hilbert spaces $\{L^2(M_n),
\gamma_{m,n}\}$ so the Hilbert space direct limit
\begin{equation*}
L^2(M_\infty) := \varinjlim \{L^2(M_n), \gamma_{n,m}\}
\end{equation*}
is well defined. Denote by $\gamma_n$ the canonical isometric embedding
$L^2(M_n)\hookrightarrow L^2(M_\infty)$. As $\gamma_{n,m}$ intertwines $L_n$
and $L_m$ it follows that we have a well defined unitary representation of
$U_\infty:=\varinjlim U_n$ on $L^2(M_\infty )$ given by: If $x\in U_n$ and $f
=\gamma_n(f_n)\in L^2(M_\infty)$, then $L_\infty (x)f=\gamma_n(L_n(x)f_n)$.
Then $\gamma_n$ is a unitary $U_n$-map and according to \cite{JWCompact}, $L^2(M_\infty)$
is a multiplicity free representation of $U_\infty$. We skip the details as
they will not be needed here.

For simplicity write
\begin{equation*}
e_n(t, \mu) := e^{t\langle \mu+2\rho_n, \mu\rangle }\, ,\quad \mu\in
\Lambda^+_n\, .
\end{equation*}
Next, we define a isometric embedding $\delta_{n,m} : \mathcal{H}_t(M_{n%
\mathbb{C}}) \hookrightarrow \mathcal{H}_t(M_{m\mathbb{C}})$ for the Fock
spaces using Theorem \ref{th-SeqSp} and such that the diagram
\begin{equation}  \label{eq-diagram}
\begin{diagram} \node{L^2(M_n)}\arrow{e,t}{\gamma_{n,m}}
\arrow{s,l}{H_{t,n}} \node{L^2(M_m)} \arrow{s,r}{H_{t,m}}\\
\node{\mathcal{H}_t(M_{n\C})}\arrow{e,t}{\delta_{n,m}}
\node{\mathcal{H}_t(M_{m\C})} \end{diagram}
\end{equation}
commutes. This forces us to define $\delta_{n,m}$ by
\begin{equation}
\delta_{n,m}\left(\sum_{\mu\in\Lambda^+_n}d(\mu )\widetilde{\pi}_{\mu }^{a(\mu)}\right)
:=\sum_{\mu\in\Lambda^+_n}d (\iota_{n,m}(\mu )) \sqrt{\frac{d(\mu )}{%
d(\iota_{n,m}(\mu ))}} \frac{e_n(t,\mu )}{e_m(t,\iota_{n,m}(\mu ))}
\widetilde\pi_{\iota_{n,m}(\mu )}^{a(\mu )}\, .
\end{equation}
Here we use the notation from Theorem \ref{th-SeqSp} and view $%
V_\mu\subseteq V_{\iota_{n,m}(\mu )}$ so $a (\mu )\in V_{\iota_{n,m}(\mu )}$.

\begin{lemma}
If $m>n$ then $\delta_{n,m}:\mathcal{H}_t(M_{n\mathbb{C}})\to \mathcal{H}%
_t(M_{m\mathbb{C}})$ is an isometric $U_n$-map and the diagram (\ref%
{eq-diagram}) commutes. Furthermore, if $n\le m\le p$ then $%
\delta_{n,p}=\delta_{m,p}\circ \delta_{n,p}$.
\end{lemma}

\begin{proof}
Write $\nu = \iota_{n,m}(\mu )$. We have
\begin{eqnarray*}
\|\delta_{n,m}(\sum_{\mu\in\Lambda^+_n}d(\mu )\widetilde{\pi}_{\mu
}^{a(\mu)})\|_{m,t}^2&=& \sum_{\mu\in\Lambda^+_n}d (\nu ) e_m(2t,\nu ))
\frac{d(\mu )}{d(\nu)} \frac{e_n(2t,\mu )}{e_m(2t,\nu )} \|a(\mu )\|^2_{\nu}
\\
&=&\sum_{\mu\in\Lambda^+_n} d(\mu )e_n(2t,\mu ) \|a(\mu )\|^2_\mu \\
&=&\|\sum_{\mu\in\Lambda^+_n}d(\mu )\widetilde{\pi}_{\mu
}^{a(\mu)})\|_{n,t}^2\, .
\end{eqnarray*}
Theorem \ref{th-SeqSp} implies that $\delta_{n,m}:\mathcal{H}_t(M_{n%
\mathbb{C}})\to \mathcal{H}_t(M_{m\mathbb{C}})$ is an unitary $U$%
-isomorphism onto its image.
\end{proof}

The following is now clear from the universal mapping property of the direct limit
of Hilbert spaces, see \cite{Natarajan}:

\begin{theorem}
\label{th-result1} Let $L^2(M_\infty) := \varinjlim \{L^2(M_n),
\gamma_{n,m}\}$ as before and $\mathcal{H}_t(M_{\infty\mathbb{C}}) :=
\varinjlim \mathcal{H}_t(M_{n\mathbb{C}})$ in the category of Hilbert spaces
and isometric embeddings. Then there exists a unique unitary isomorphism $%
H_{t,\infty} : L^2(M_\infty ) \to \mathcal{H}_t(M_{\infty\mathbb{C}})$ such
that the diagram
\begin{equation*}
\begin{diagram}
\node{\dots}\arrow{e}
\arrow{s}
\node{L^2(M_n)}\arrow{e,t}{\gamma_{n+1,n}}
\arrow{s,l}{H_{t,n}}
\node{L^2(M_{n+1})}\arrow{e}
\arrow{s,r}{H_{t,n+1}}
\node{\dots}\arrow{e}
\arrow{s}
\node{L^2(M_\infty)}
\arrow{s,l}{H_{t,\infty}}\\
\node{\dots}\arrow{e}
\node{\mathcal{H}_t(M_{n\C})}
\arrow{e,t}{\delta_{n+1,n}}
\node{\mathcal{H}_t(M_{n+1,\C})}\arrow{e}
\node{\dots}\arrow{e}
\node{\mathcal{H}_t(M_{\infty\C})}
\end{diagram}
\end{equation*}
commutes. In particular, if $\delta_n : \mathcal{H}_t(M_{n\mathbb{C}})\to
\mathcal{H}_t(M_{m\mathbb{C}})$ and $\gamma_n : L^2(M_n)\to L^2(M_\infty)$
are the canonical embedding then $\delta_n\circ H_{t,n}=H_{t,\infty}\circ
\gamma_n$.
\end{theorem}

\section{The Segal-Bargman Transform on the Direct Limit of $%
\{L^2(M_n)^{K_n}\}_n$}

\noindent We continue using the notations as in the previous section. We
pointed out earlier that the $U_n$-embedding $V_\mu \hookrightarrow
V_{\iota_{n,m}(\mu )}$ does not map $V_\mu^{K_n}$ into $V_{\iota_{n,m}(\mu
)}^{K_m}$. This implies that
\begin{equation*}
\gamma_{n,m}(L^2(M_n)^{K_n}) \not \subset L^2(M_m)^{K_m}
\end{equation*}
and $\gamma_{n,m}(\psi_\mu )\not=\psi_{\iota_{n,m}(\mu )}$. Therefore, to
describe the limit of the heat transform on the level of $K$-invariant
functions, a new embedding is needed. We therefore define $%
\eta_{n,m}:L^2(M_n)^{K_n}\to L^2(M_m)^{K_m}$ by
\begin{equation*}
\eta_{m,n}(f) := \sum_{\mu \in\Lambda^+} d(\iota_{n,m}(\mu)) \sqrt{\frac{%
d(\mu)}{d(\iota_{n,m}(\mu))}} \widehat{f}(\mu ) \psi_{\iota_{n,m}(\mu)}\, .
\end{equation*}
All $\eta_{m,n}$ are linear and it follows from Theorem \ref%
{SphericalPlancherel} that $\eta_{m,n}$ is an isometric embedding.
Furthermore, if $n \leq m \leq p$, then $\eta_{n,p} = \eta_{m,p} \circ
\eta_{n,m}$. Therefore $\{L^2(M_n)^{K_n}, \eta_{n,m}\}$ is a direct system
of Hilbert spaces. Define (by abuse of notation)
\begin{equation*}
L^2 (M_\infty)^{K_\infty} := \varinjlim L^2(M_n)^{K_n}
\end{equation*}
in the category of Hilbert spaces and isometric embeddings from the above
direct system. Denote by $\eta_n :L^2(M_n)\to L^2 (M_\infty)^{K_\infty}$ the
resulting canonical embedding.

For $m \geq n$ define
\begin{equation*}
\phi_{n,m}:\mathcal{H}_t(M_{n\mathbb{C}})^{K_{n\mathbb{C}}}\longrightarrow
\mathcal{H}_t(M_{m\mathbb{C}})^{K_{m\mathbb{C}}}
\end{equation*}
by
\begin{equation*}
\sum_{\mu\in\Lambda^+} d(\mu )a (\mu)\widetilde{\psi}_{\mu } \longmapsto
\sum_{\mu\in\Lambda^+} d(\iota_{n,m}(\mu))\sqrt{\frac{d(\mu)}{%
d(\iota_{n,m}(\mu ) )}} \frac{e_n(t, \mu )}{e_m(t, \iota_{n,m}(\mu))}\,
a(\mu ) \widetilde{\psi}_{\iota_{n,m}(\mu )}\, .
\end{equation*}
By Theorem \ref{th-SeqSp} it follows that $\phi_{n,m}$ is an isometric $U_n$%
-embedding: If
\[F := \sum_{\mu\in\Lambda^+} d(\mu)a(\mu )\widetilde{\psi}_{\mu} \in \mathcal{H}_t(M_{n\mathbb{C}})^{K_{n\mathbb{C}}}\]
then:
\begin{eqnarray*}
\|\phi_{n,m}(F)\|^2_{m,t} &=&\sum_{\mu\in\Lambda^+} d(\iota_{n,m}(\mu ))\,
e_m(2t, \iota_{n,m}(\mu ))\, \left|\sqrt{\frac{d(\mu)}{d(\iota_{n,m}(\mu ))}}%
\, a(\mu ) \frac{e_n(t, \mu )}{e_m(t, \iota_{n,m}(\mu))}\, \right|^2 \\
&=& \sum_{\mu\in\Lambda^+} d(\mu ) |a_t(\mu)|^2 e_n(2t, \mu ) \\
&=& \|F\|^2_{n,t}\, .
\end{eqnarray*}
Finally, it is easy to verify that if $n \leq m \leq p$, then $\phi_{n,p} =
\phi_{m,p} \circ \phi_{n,m}$.

Thus $\mathcal{H}_{t}(M_{\infty\mathbb{C}})^{K_{\infty\mathbb{C}}} := \varinjlim \{\mathcal{H}%
_t(M_{n\mathbb{C}})^{K_{n\mathbb{C}}}$ is well defined in the category of
Hilbert spaces and isometric embeddings.

\bigskip

\begin{lemma}
\label{lemmacommutative} For $m \geq n$, the following diagram is
commutative.
\begin{equation*}
\begin{diagram}
\node{L^2(M_n)^{K_n}}\arrow{e,t}{\eta_{n,m}}
\arrow{s,l}{H_{t,n}}
\node{L^2(M_m)^{K_m}}
\arrow{s,r}{H_{t,m}}\\
\node{\mathcal{H}_t(M_n^\C)^{K_n^\C}}\arrow{e,t}{\phi_{n,m}}
\node{\mathcal{H}_t(M_m^\C)^{K_m^\C}}
\end{diagram}
\end{equation*}
\end{lemma}

\begin{proof}
Let $f=\sum_{\mu\in\Lambda^+} d(\mu )\widehat{f}(\mu )\psi_{\mu } \in
L^2(M_n)^{K_n}$. Then
\[H_{t,n}(f) = \sum_{\mu } d(\mu )e_n(- t, \mu )\widehat{f}(\mu )\psi_{\mu}\]
and
\begin{eqnarray*}
\phi_{n,m}(H_{t,n}(f)) &=&\sum_{\mu\in\Lambda^+} d(\mu )\sqrt{\frac{d(\mu)}{%
d(\iota_{n,m}(\mu))}}\, \, \frac{e_n(t, \mu )}{e_m(t, \iota_{n,m}(\mu ))}\,
\widehat{f}(\mu )e_n(t, \mu)^{-1}\widetilde{\psi}_{\iota_{n,m}(\mu )} \\
&=& \sum_{\mu\in\Lambda^+} d(\mu )\sqrt{\frac{d(\mu)}{d(\iota_{n,m}(\mu ))}}%
\, \widehat{f}(\mu ) e_m(t, \iota_{n,m}(\mu) )^{-1}\widetilde{\psi}%
_{\iota_{n,m}(\mu )} \\
&=& H_{t,m}(\eta_{n,m}(f))\, .
\end{eqnarray*}
\end{proof}

Using the universal mapping property of the direct limit of Hilbert spaces
as in Theorem 6.2, we obtain the following:

\begin{theorem}
\label{result} There exists a unique unitary isomorphism
\[S_{t,\infty} :
L^2(M_\infty)^{K_\infty} \to \mathcal{H}_t (M_{\infty\mathbb{C}})^{K_{\infty\mathbb{C}}}\]
such that
the diagram
\begin{equation*}
\begin{diagram}
\node{\dots}\arrow{e}
\arrow{s}
\node{L^2(M_n)^{K_n}}\arrow{e,t}{\eta_{n+1,n}}
\arrow{s,l}{H_{t,n}}
\node{L^2(M_{n+1})^{K_{n+1}}}\arrow{e}
\arrow{s,r}{H_{t,n+1}}
\node{\dots}\arrow{e}
\arrow{s}
\node{L^2(M_\infty)^{K_\infty}}
\arrow{s,l}{S_{t,\infty}}\\
\node{\dots}\arrow{e}
\node{\mathcal{H}_t(M_n^\C)^{K_n^\C}}
\arrow{e,t}{\phi_{n+1,n}}
\node{\mathcal{H}_t(M_{n+1}^\C)^{K_{n+1}^\C}}\arrow{e}
\node{\dots}\arrow{e}
\node{\mathcal{H}_t (M_{\infty\mathbb{C}})^{K_{\infty\mathbb{C}}}}
\end{diagram}
\end{equation*}
commutes. Furthermore $\eta_\infty\circ S_{t\infty}=\phi_n\circ H_{t,n}$.
\end{theorem}

\end{document}